\documentclass[a4paper,12pt]{amsart}
\usepackage{geometry}
\usepackage{euscript,verbatim}
\usepackage[psamsfonts]{amssymb}
\usepackage[usenames]{color}
\usepackage{graphicx}
\usepackage[colorlinks,linkcolor=red,anchorcolor=blue,citecolor=blue]{hyperref}
\usepackage{mathtools}
\usepackage{amsmath}
\newtheorem{theorem}{Theorem}[section]
\newtheorem{lemma}[theorem]{Lemma}

\newtheorem{proposition}[theorem]{Proposition}
\newtheorem{corollary}[theorem]{Corollary}

\newtheorem{remark}{Remark}[section]
\newtheorem{example}[theorem]{Example}

\def\R{\mathbb{R}}
\def\N{\mathbb{N}}
\def\Z{\mathbb{Z}}
\def\Q{\mathbb{Q}}

\usepackage{cancel}
\usepackage[normalem]{ulem}

\begin{document}

\title[Non-Archimedean Koksma Theorems and Dimensions of Exceptional Sets]{Non-Archimedean Koksma Theorems and Dimensions of Exceptional Sets
}

\date{} 

\author[A. H. FAN]{Aihua Fan}
\address[A. F. FAN]
{LAMFA, UMR 7352, University of Picardie, 33 Rue Saint Leu, 80039, Amiens, France and  Wuhan Institute for Math \& AI, Wuhan University, Wuhan 430072, China}
\email{ai-hua.fan@u-picardie.fr}

\author[S. L. FAN]{Shilei Fan}
\address[S. L. FAN]
{School of Mathematics and Statistics, and Key Lab NAA--MOE, Central China Normal University, Wuhan 430079, China}
\email{slfan@mail.ccnu.edu.cn}

\author[H. F. YE]{Hanfei Ye}
\address[H. F. YE]
{School of Mathematics and Statistics \& Hubei Key Laboratory of Mathematical Sciences, Central  China Normal University, Wuhan, 430079, China }
\email{yehanfei15@mails.ucas.ac.cn}

\thanks{ The authors were supported by NSFC (grants No. 12331004  and No.  12231013) and NSF of Xinjiang Uygur Autonomous Region (Grant No. 2024D01A160).}

\begin{abstract}
We establish a non-Archimedean analogue of Koksma's theorem.
For a local field $\mathcal{F}$ of characteristic zero, we prove that the sequence $([\alpha x^n])$ is uniformly distributed in the valuation ring $\mathcal{O}$ for almost every $x$ with $|x|_\mathfrak{p}>1$.
In the case of positive characteristic, $([x^n])$ fails to be uniformly distributed, but it becomes $\mu^*$-uniformly distributed for some weighted measure $\mu^*$.
These results are derived from a general metric theorem for sequences generated by expanding scaling maps.
On the other hand,  we demonstrate that the exceptional set of parameters $x$ for which these sequences are not uniformly distributed is large (i.e. having full Hausdorff dimension) and share a rich $q$-homogeneous fractal structure.

\end{abstract}
\subjclass[2010]{}
\keywords{Uniform Distribution, Non-Archimedean Fields, Koksma's Theorem, Hausdorff Dimension.}
\maketitle

\section{Introduction}
A sequence of real numbers $(x_n)$ is said to be equi-distributed or uniformly distributed modulo 1 (u.d. mod 1 for short), if the sequence of fractional parts $(\{x_n\})$ is uniformly distributed in $\mathbb{T}=\mathbb{R}/\mathbb{Z}$, which means that for every pair $a, b$ of real numbers with $0 \leq a < b \leq 1$ we have
\[\lim_{N \rightarrow \infty}\frac{1}{N}\#\{1\leq n\leq N: x_n \in [a,b)\}=b-a.\]
In 1916, Weyl laid a criterion for determining whether a sequence is u.d. mod 1 or not \cite{Weyl1916}. Weyl's theorem states that the following are equivalent:
\begin{itemize}
\item[(1)] a sequence $(x_n)_{n\geq 1} \subset \mathbb{T} $ is uniformly distributed;
\item[(2)] for every continuous function $f\in C(\mathbb{T})$ we have 
\[\lim_{N\to \infty }\frac{1}{N} \sum_{n=1}^{N}f(x_n)=\int_\mathbb{T}f(x)dx;\]
\item[(3)] for $f$ belonging to a total set of continuous functions (i.e. their linear combinations are dense in $C(\mathbb{T})$), we have the same conclusion of (2);
\item[(4)]  for each $\ell \in \Z\setminus\{0\} $,  we have 
\[\lim_{N\to \infty }\frac{1}{N} \sum_{n=1}^{N}e^{2\pi i \ell x_n}=0.\]
\end{itemize}

In general, let $X$ be a compact Hausdorff space and $\mu$ be a Borel probability measure on $X$. A sequence of points $(x_n)\subset X$ is said to be $\mu$-uniformly distributed ($\mu$-u.d. for short)
if the probability measures $\frac{1}{N}\sum_{n=1}^{N} \delta_{x_n}$ converge to $\mu$ in weak$^*$-topology as $N\to \infty$, where $\delta_a$ denotes the Dirac measure at the point $a$.
If $X$ is a compact Abelian group and $\mu$ is the Haar measure, Weyl's theorem remains true.
We refer \cite{KN1974} for such generalization. 
\medskip

One of the basic results concerning the u.d. mod 1 for sequences in $\mathbb{R}$ is Koksma's equi-distribution theorem \cite{Kok1935}, which states that for Lebesgue-almost every $\alpha> 1$, the geometric progression $(\alpha^n)_{n\geq 1}$ is uniformly distributed mod $1$.
Such sequences with a “typical” value of $\alpha$ was proposed by Knuth \cite{Knuth1998} as sequences showing strong pseudo-randomness properties.
\medskip

In the present paper we will prove a non-Archimedean version of Koksma's theorem.
We are concerned with the integral parts of the powers $x^n$ ($x$ is in a local field $\mathcal{F}$ with $|x|_\mathfrak{p}>1$) which fall into the valuation ring $\mathcal{O}$ of $\mathcal{F}$, which is considered as an additive group.
Bertrandias \cite{Bert1967} considered a $p$-adic version of Koksma's theorem, concerning the fractional parts of $x^n$,  which fall into the unit interval $[0,1)$ in $\mathbb{R}$. 
\medskip

Let us first recall some notations before  presenting our setting and main results.
For a non-Archimedean local field $\mathcal{F}$, denote its valuation ring by $\mathcal{O}$ and its maximal ideal by $\mathfrak{p}$.
Fix a prime element $\pi$ and choose a complete set 
$C \subset \mathcal{O}$ of representatives of $\mathcal{O}/\mathfrak{p}$ with $0 \in C$.
A typical element \(x \in \mathcal{F}\) can be written in the form
\begin{equation}\label{eq:pnumber}
    x = \sum_{n = v}^{\infty} c_n \pi^{\,n},
\end{equation}
where \(v \in \mathbb{Z}\) and the coefficients \(c_n \in C\) satisfy \(c_v \neq 0\).

This series converges in non-Archimedean norm.
The non-Archimedean norm of $x$ is defined by $|x|_\mathfrak{p}:=q^{-v}$, where $q=\#\mathcal{O}/\mathfrak{p}$.
Let \[\{x\}:=\sum_{n=v}^{-1}c_n \pi^{n},\] called the fraction part of $x$, and let   \[[x]:=\sum_{n = 0}^{\infty}c_n \pi^{n},\] called the integral part of $x$.
\medskip

In this setting of non-Archimedean local field $\mathcal{F}$,  Weyl's criterion states as follows. A sequence $(x_n)\subset \mathcal{O}$ is uniformly distributed in $\mathcal{O}$ if for every $a$ in $\mathcal{O}$ and every $k \in \N$, we have 
\[\lim _{N\to \infty}\frac{1}{N}\#\{1\leq n \leq N: x_n \in  D(a, 1/q^k)\}=1/q^k,\]
where $D(a, 1/q^k)=\{x\in \mathcal{F}: |x-a|_\mathfrak{p}\leq 1/q^k\}$ is the disk of radius $1/q^k$ centered at $a$.
\medskip

Our first result is the following non-Archimedean version of Koksma's equi-distribution theorem. 
\begin{theorem}\label{thm-koksma}
Fix $\alpha \in \mathcal{F}$ with $\alpha \neq 0$.
\begin{itemize}
\item[{\rm(a)}] If ${\rm char}\mathcal{F}=0$, then the sequence $\bigl([\alpha x^n]\bigr)_{n\geq 1}$ is uniformly distributed in $\mathcal{O}$ for almost all $x\in \mathcal{F}$ with $|x|_\mathfrak{p}>1$;
\item [{\rm(b)}] If ${\rm char}\mathcal{F}=p>0$, then the sequence $\bigl([\alpha x^n]\bigr)_{n\geq 1,p\nmid n}$ {\rm(}the subsequence of $\bigl([\alpha x^n]\bigr)$ along $n$'s which have no factor $p${\rm)} is uniformly distributed in $\mathcal{O}$ for almost all $x\in \mathcal{F}$ with $|x|_\mathfrak{p}>1$.
\end{itemize}
\end{theorem}

The conclusion of Theorem \ref{thm-koksma} (b) does not hold in general for the whole sequence $\bigl([\alpha x^n]\bigr)$ (see  Theorem  \ref{thm-char=p}).
Theorem \ref{thm-koksma} will be proved as a consequence of the following general theorem about scaling maps.
Let $\Omega$ be a disk in $\mathcal{F}$. A map $f: \Omega \to \mathcal{F}$ is said to be  \textit{scaling} of \textit{scaling ratio} $q^\lambda$ for some $\lambda \in \mathbb{Z}$ if 
\[|f(x)-f(y)|_\mathfrak{p}= q^\lambda |x-y|_\mathfrak{p}, \quad  \forall  x, y \in \Omega.\]
If $\lambda \ge 1$, we say that $f$ is an \textit{expanding scaling map}. A map $f: U \to \mathcal{F}$ defined on an open set $U$ is said to be \textit{locally scaling} if for any $x\in U$, $f$ is scaling in some neighborhood of $x$.

\begin{theorem}\label{thm-main}
Let $\Omega\subset \mathcal{F}$ be a disk in $\mathcal{F}$ and let $(f_n)$ be a sequence of scaling maps from $\Omega$ into $\mathcal{F}$ of scaling ratios $(q^{\lambda_n})$. Suppose
\begin{itemize} 
\item[(a)] $q^{\lambda_n} \geq {\rm diam}(\Omega)^{-1}$ for all $n\ge 1$;
\item[(b)] $\sum_{N=1}^{+\infty} \frac{\# K_{N}}{N^{3}} <\infty$ where
$$
K_N=\{(n, m): 1\le n, m\le N ,\; \lambda_n=\lambda_m\}.
$$
\end{itemize}
Then $([f_{n}(x)])_{n\geq1}$ is uniformly distributed in $\mathcal{O}$ for almost all $x\in \Omega$.
\end{theorem}

The following result, which has a counterpart in $\R$ (cf. \cite[p.35]{KN1974}), is also a consequence of Theorem \ref{thm-main}, an immediate consequence.
\begin{corollary}\label{thm-koksma2}
Fix $\beta \in \mathcal{F}$ with $|\beta|_\mathfrak{p}> 1$. Then the sequence $\bigl([\beta^n x]\bigr)_{n\geq 1}$ is uniformly distributed in $\mathcal{O}$ for almost all $x\in \mathcal{F}$.
\end{corollary}

Our proof of Theorem \ref{thm-main}  is different  from  the classical proof of Koksma's Theorem.
We will apply the famous Theorem of Davenport-Erd\"os-LeVeque (see Lemma \ref{thm-DEL}).
\medskip

When ${\rm char}\mathcal{F}=p>0$ the sequence $\bigl([x^n]\bigr)_{n\geq 1}$ is not Haar-u.d. but is almost surely \ u.d. under a constructed weighted measure $\mu^*$.

Denote the normalized Haar measure of $\mathcal{O}$ by $\mu$.
For $k\geq 1$, let
\[\mathcal{S}_k:=\left\{x= \sum_{i=v_\mathfrak{p}(x)}^\infty c_i \pi^i \in \mathcal{F}:\ c_i\in \mathbb{F}_q;\ c_i=0 \text{ if } p^k\nmid i \right\}.\]
Let $\widetilde{\mathcal{S}_k}=\mathcal{S}_k\cap\mathcal{O}$ be an additive closed subgroup of $\mathcal{O}$ and $\mu_k$ be the Haar measure of $\mathcal{S}_v$ normalized by $\mu_k(\widetilde{\mathcal{S}_k})=1$.
A Borel measure $\mu^*$ on $\mathcal{O}$ is defined as
\[\mu^*(D):= \left(1-\frac{1}{p}\right)\left(\mu(D)+\sum_{k=1}^\infty p^{-k}\mu_k\left(D\cap\widetilde{\mathcal{S}_k}\right)\right)\]
for any disk $D$ in $\mathcal{O}$.

\begin{theorem}\label{thm-char=p}
Suppose ${\rm char}\mathcal{F}=p>0$.  
\begin{itemize}
\item[{\rm(a)}] The sequence $\bigl([x^n]\bigr)_{n\geq 1}$ is not $\mu$-u.d. in $\mathcal{O}$ for all $x\in \mathcal{F}$ with $|x|_\mathfrak{p}>1$. 
\item [{\rm(b)}]For each positive integer $k$, the subsequence $\bigl([x^n]\bigr)_{n\geq 1,p^k\parallel n}$ {\rm (}the subsequence of $\bigl([x^n]\bigr)$ along $n$'s which are exaclty divided by $p^k${\rm )} is $\mu_k$-u.d. for Haar-almost all $x\in \mathcal{F}$ with $|x|_\mathfrak{p}>1$.
\item [{\rm(c)}] The sequence $\bigl([x^n]\bigr)_{n\geq 1}$ is $\mu^*$-u.d. for Haar-almost all $x\in \mathcal{F}$ with $|x|_\mathfrak{p}>1$.
\end{itemize}
\end{theorem}

The reason for the dichotomy between Theorem \ref{thm-koksma} and Theorem \ref{thm-char=p} is that in the case of ${\rm char}\mathcal{F}=p$, the map $f_n(x)=\alpha x^n$ is locally scaling when $p\nmid n$ (see Lemma \ref{lem-scaling}), while the subsequence $\bigl([x^n]\bigr)_{p^k\mid n}$ falls into the zero measure compact set $\widetilde{\mathcal{S}_k}$.
In the case of ${\rm char}\mathcal{F}=0$, the map $f_n(x)=\alpha x^n$ is locally scaling.
\medskip

On the other hand, there are numerous works aiming to study sets of points that exhibit highly non-uniform distribution behavior in the real setting.
In 1957, Erd\"os and Taylor \cite{ET1957} studied the distribution of the sequence $(\{n_k x\})$. They showed that if the Hadamard lacunary condition $\inf_kn_{k+1}/n_k\geq q>1$ is satisfied, then the exceptional set
\[
\mathcal{B}: = \{x\in\R: (\{n_k x\}) \text{ is not u.d. in } [0,1]\}
\]
has full Hausdorff dimension. 
This result was generalized to the multidimensional case by Fan in 1993 \cite{Fan1993}, where  Riesz product measures are used as tool to prove that $\mathcal{B}$ has full Hausdorff dimension.
In 1980, Pollington \cite{Pol1980} proved that for any $\epsilon>0$, the set of highly biased points
\[
\{x>1 : \{x^n\}<\epsilon,\ \forall n\}
\]
is of full Hausdorff dimension.  
In 2014, Kahane \cite{Kahane2014} established an inhomogeneous analogue: for any sequence of real numbers $(b_n)$,
\[
\dim_{\mathcal{H}}\{x>1 : \Vert x^n-b_n\Vert <\epsilon,\ \forall n\}=1.
\]
Here $\Vert x \Vert=\min \{\{x\},1-\{x\}\}$.
Baker \cite{Baker2015} showed in 2015 that if $(r_n)$ is a sequence of numbers such that
\[\lim_{n\to\infty}(r_{n+1}-r_n)=\infty,\]
then
\[
\dim_{\mathcal{H}}\{x>1 : \lim_{n\to\infty}\Vert x^{r_n}-b_n \Vert =0\}=1.
\]

\medskip

In the non-Archimedean case, in contrast to the uniform distribution phenomena established in Theorems~\ref{thm-koksma}--\ref{thm-main}, there also exist large sets of parameters $x$ (of full or maximal Hausdorff dimension) for which the associated sequences exhibit non-uniform distributional or even highly biased behavior.
\medskip

For scaling maps $(f_n)$ of strictly increasing scaling ratios,  the exceptional set is of full Hausdorff dimension, as the following theorem shows.
For an element $x$ in $\mathcal{F}$, we use $[x]_0$ to denote its $0$-digit, name $[x]_0=c_0$ if $x$ takes the form \eqref{eq:pnumber}.  
We say that a sequence taking values in the set of symbols $\{a_0,a_1, \cdots, a_{m-1}\}$ is uniformly distributed if it has frequency $\frac{1}{m}$ for each symbol.

\begin{theorem}\label{thm-fulldimscaling}
Let $\Omega$ be a disk in $\mathcal{F}$.  Let $f_{n}:\Omega \to \mathcal{F}$ is a  scaling map of scaling ratio $q^{\lambda_n}$ for each $n\ge 1$. Suppose that $(\lambda_n)$ is strictly increasing. Then
\[
\dim_{\mathcal{H}}\bigl\{ x \in \Omega : \bigl([f_n(x)]_0 \bigr) \text{ is not uniformly distributed in } \mathcal{O}/\mathfrak{p} \bigr\} = 1.
\]
Consequently,
\[
\dim_{\mathcal{H}}\bigl\{ x \in \Omega : \bigl([f_n(x)] \bigr) \text{ is not uniformly distributed in } \mathcal{O} \bigr\} = 1.
\]
\end{theorem}

We have an immediate corollary about the special sequence $(\beta^n x)$.

\begin{corollary}\label{cor-fulldim}
Fix $\beta \in \mathcal{F}$ with $|\beta|_\mathfrak{p}>1$. Then 
\[
\dim_{\mathcal{H}}\bigl\{ x \in\mathcal{F} : \bigl([\beta^n x]\bigr) \text{ is not uniformly distributed in } \mathcal{O} \bigr\} = 1.
\]
\end{corollary}

The map $f_n(x)=\alpha x^n$ is locally scaling when ${\rm char}\mathcal{F}=0$ (see Lemma \ref{lem-scaling}).  
However, the scaling ratios are not strictly increasing. In fact, one cannot even extract a subsequence of density $1$ with increasing scaling ratios. Therefore, Theorem~\ref{thm-fulldimscaling} cannot be applied directly to the sequence $([\alpha x^n])$. Nevertheless, the case of $([\alpha x^n])$ can still be treated by alternative methods.

\begin{theorem}\label{cor-fulldim2}
Fix $\alpha \in \mathcal{F}$ with $\alpha \neq 0$. Suppose ${\rm char}\mathcal{F}=0$, then
\[
\dim_{\mathcal{H}}\{ x\in \mathcal{F} : |x|_\mathfrak{p}>1,\; \bigl([\alpha x^n]\bigr) \text{ is not uniformly distributed in } \mathcal{O} \}=1.
\]
\end{theorem}

In order to prove Theorem \ref{cor-fulldim2}, we are led to deal with the set of points $x$ such that $([\alpha x^n])$ are highly biased (far from uniformly distributed).
We show that this set possesses a $q$-homogeneous structure (see Section \ref{sec:p-homogeneous}), and then its Hausdorff dimension can be explicitly computed (see Lemma \ref{thm-main-ebp}).
Here, by a $q$-homogeneous set we mean a special type of homogeneous Moran type set contained in a disk.
\medskip

In Section 2, basic notions and notations of local fields and $q$-homogeneous sets are provided.
Theorem \ref{thm-main} is proved in Section 3 and Theorem \ref{thm-koksma} and Theorem \ref{thm-char=p} are proved in Section 4.
Theorem \ref{thm-fulldimscaling} and Theorem \ref{cor-fulldim2} are proved respectively in Section 5 and Section 6.
In Section 7, we talk about some metrical results on the distribution based on Lemma \ref{thm-main-ebp}.

\section{Preliminary}

We start with a quick review of local fields. Then we introduce the concept of $q$-homogeneous sets, which is a special type of homogeneous Moran sets.

\subsection{Local fields}
Recall that a local field is a field that is complete with respect to a discrete valuation and has a finite residue class field.

The followings are some notions about a non-Archimedean local field $\mathcal{F}$ in this paper:\\
$v_\mathfrak{p}: \mathcal{F} \rightarrow \Z \cup \{+\infty\}$ is the the normalized exponential valuation satisfying that
\begin{itemize}
    \item[(i)] \(v_\mathfrak{p}(x) = +\infty\) only when \(x = 0\),
    \item[(ii)] \(v_\mathfrak{p}(xy) = v_\mathfrak{p}(x) + v_\mathfrak{p}(y)\),
    \item[(iii)] \(v_\mathfrak{p}(x+y) \geq \min\{ v_\mathfrak{p}(x), v_\mathfrak{p}(y)\}\);
\end{itemize}
$|\cdot|_\mathfrak{p}=q^{-v_\mathfrak{p}(\cdot)}$ is the normalized absolute value;\\
$\mathcal{O}:=\{x\in \mathcal{F}:v_\mathfrak{p}(x)\geq 0\}=\{x\in \mathcal{F}:|x|_\mathfrak{p}\leq 1\}$ is the valuation ring;\\
$\mathfrak{p}:=\{x\in \mathcal{F}:v_\mathfrak{p}(x)\geq 1\}=\{x\in \mathcal{F}:|x|_\mathfrak{p}< 1\}$ is the only maximal ideal of $\mathcal{O}$;\\
$\mathcal{O}/\mathfrak{p}$ is the residue class field and it is finite;\\
$q:=\#\mathcal{O}/\mathfrak{p}$;\\
$\pi$ is a fixed prime element satisfying that $v_\mathfrak{p}(\pi)=1$.

The absolute value of a local field satisfies the strong triangle inequality:
\[|x+y|_\mathfrak{p} \leq \max\{|x|_\mathfrak{p},|y|_\mathfrak{p}\} \qquad \forall x,y\in \mathcal{F}.\]
In particular, if $|x|_\mathfrak{p}>|y|_\mathfrak{p}$, then $|x+y|_\mathfrak{p}=|x|_\mathfrak{p}$.

Let $C\subset \mathcal{O}$ be a complete system of representatives for $\mathcal{O}/\mathfrak{p}$ with $0\in C$, then every $x\neq 0$ in $\mathcal{F}$ can be uniquely written as
\[
    x = \sum_{n = v_\mathfrak{p}(x)}^\infty c_n \pi^{n} \quad (c_n \in C, \ x_{v_\mathfrak{p}(x)} \neq 0).
\]
Let $\{x\}:=\sum_{n=v}^{-1}c_n \pi^{n}$, called the fraction part of $x$, and let   $[x]:=\sum_{n = 0}^{\infty}c_n \pi^{n}$, called the integral part of $x$.

For $k\in \Z$ and $a\in \mathcal{F}$, we denote by
\[D(a, q^{k})=\{x\in \mathcal{F}: |x-a|_\mathfrak{p}\leq q^k\} =a+\pi^{-k}\mathcal{O}\] 
the disk of radius $q^k$ centered at $a$.

The typical measure $\mu$ on $\mathcal{F}$ is defined by
\[\mu(D(a, q^{k}))={\rm diam}(D(a, q^{k}))=1/q^k\]
for any disk $D(a, q^{k})$. Its restriction on $\mathcal{O}$ is the normalized Haar measure of the additive group $\mathcal{O}$.

The definition of Hausdorff dimension in $\mathcal{F}$ refers to the natural metric $d(x,y)=|x-y|_\mathfrak{p}$ (cf. \cite{Fal2014} for Hausdorff dimension).

A non-Archimedean local field $\mathcal{F}$ is precisely $\mathbb{F}_q((\pi))$ or a finite extension of $\Q_p$, where $p$ is a prime integer, $q$ is a prime power and $\pi$ is a transcendental element (cf. \cite{JN1999}).
If ${\rm char}\mathcal{F}=0$, let $e$ be the ramification index of the extension $\mathcal{F} \mid \Q_p$. Then $[\mathcal{F}:\Q_p]=e\log_p q$.

 Recall that the \textit{$p$-adic valuation}  
\[
v_p : \mathbb{Z} \to \mathbb{Z}_{\ge 0} \cup \{\infty\}
\]
is defined as follows: for any nonzero integer $k$,  
\[
v_p(k)= \text{the largest integer } r \text{ such that } p^r \mid k,
\]
and we set $v_p(0)=\infty$.
Let $\mathcal{F}/\mathbb{Q}_p$ be a finite extension   with  ramification index  $e$.  Then for every integer $k\in\mathbb{Z}\subset \mathcal{F}$ we have
\[
v_{\mathfrak{p}}(k)=e\,v_p(k),
\]
which expresses the compatibility between the $\mathfrak{p}$-adic valuation on $\mathcal{F}$ and the usual $p$-adic valuation on $\mathbb{Q}_p$.

\subsection{\texorpdfstring{$q$}.-homogeneous subset of \texorpdfstring{$\mathcal{F}$}.}\label{sec:p-homogeneous}
We regard the field $\mathcal{F}$ as an infinite tree 
$(\mathcal{T}, \mathcal{E})$.  
The vertex set $\mathcal{T}$ consists of all closed balls in $\mathcal{F}$.  
A vertex is said to be at level $n \in \mathbb{Z}$ if the corresponding ball $B$ satisfies $\mu(B)=q^{-n}$ (note that for a closed ball $B$, $\mu(B)={\rm diam}(B)$).
The edge set $\mathcal{E}$ is defined as the collection of pairs $(B', B) \in \mathcal{T} \times \mathcal{T}$ such that
\[B' \subset B, \qquad \mu(B') = q^{-1} \mu(B).\]  
We write $B' \prec B$ to indicate this relation, and call $B'$ a \textit{son} of $B$, while $B$ is the \textit{father} of $B'$.

\begin{figure}[ht]
  \centering
  \includegraphics[width=1\textwidth]{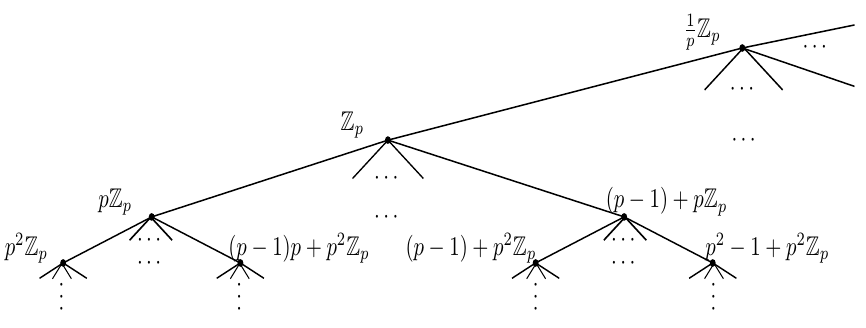}
  \caption{The field $\mathbb{Q}_p$ viewed as an infinite tree.}
  \label{Qp}
\end{figure}

Any bounded subset $\Omega \subseteq \mathcal{F}$ can be represented by a subtree $(\mathcal{T}_\Omega, \mathcal{E}_\Omega)$ of $(\mathcal{T}, \mathcal{E})$.  
Indeed, let $B^*$ be the smallest ball containing $\Omega$, which serves as the root of the subtree.  
For each $x \in \Omega$, there exists a unique sequence of nested balls 
\[
   \cdots \prec B_n \prec B_{n-1} \prec \cdots \prec B_1 \prec B_0 = B^*
\]
such that 
\[
   x \in B_n, \qquad \forall n \geq 0.
\]
We define $\mathcal{T}_\Omega$ to be the set of all such balls contained in $B^*$ that intersect $\Omega$,  
and $\mathcal{E}_\Omega$ to be the set of all edges $B_i \prec B_{i+1}$ arising in this way.  

A subtree $(\mathcal{T}', \mathcal{E}')$ is called \textit{homogeneous} if the number of sons of each $B \in \mathcal{T}'$ depends only on $\mu(B)$.  
If this number is always either $1$ or $q$, we say that $(\mathcal{T}', \mathcal{E}')$ is a \textit{$q$-homogeneous tree}.  
Accordingly, a bounded set is called \textit{homogeneous} (resp. \textit{$q$-homogeneous}) if its associated tree is homogeneous (resp. $q$-homogeneous).  

For a bounded homogeneous set $\Omega$, we define the set of \textit{branch levels} by
\[
I_{\Omega} = \Big\{ n \in \mathbb{Z} : \exists\, B \in \mathcal{T}_\Omega \text{ with } \mu(B) = q^{-n} 
\text{ and } B \text{ has } q \text{ sons} \Big\}.
\]

Intuitively, a $q$-homogeneous subset of $\mathcal{O}$ with branch level set $\mathbb{N}\setminus\{\lambda_n\}$ can be described as follows:  
at each level $\lambda_n$, exactly $q-1$ branches are removed, while at all other levels every branch is retained.

\begin{figure}[ht]
 	\centering
 	\includegraphics[width=0.7\linewidth]{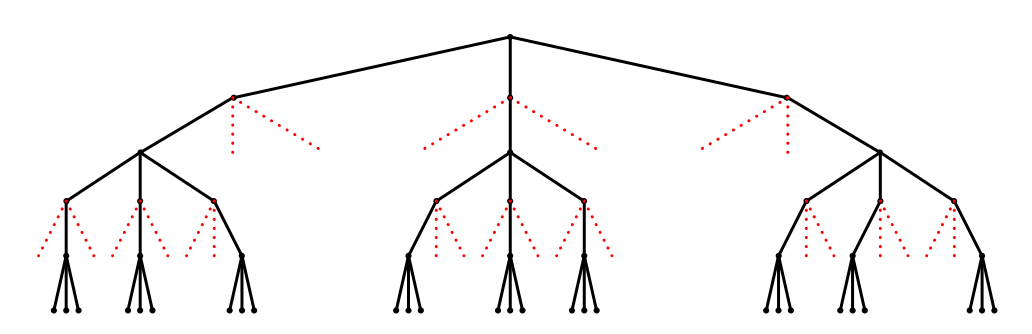}
 	\caption{A $3$-homogeneous set with $0,2,4\in I_{\Omega}$ .}
 	\label{fig:1}
 \end{figure}

\section{Proof of Theorem \ref{thm-main}}

The proof of Theorem \ref{thm-main} is based on the following two facts. The first one is that for a scaling map $f:\mathcal{O} \rightarrow \mathcal{F}$, the associated map $\tilde{f}:\mathcal{O} \rightarrow \mathcal{O}$ defined by $\tilde{f}(x) = [f(x)]$ preserves the Haar measure (see Lemma \ref{lem-1}). The second one is that for any disk $D$ in $\mathcal{O}$, if two scaling maps $f,g:\mathcal{O} \rightarrow \mathcal{F}$ have largely different scaling ratios, then $1_D \circ \tilde{f}$ and $1_D \circ \tilde{g}$ are non-correlated (see Lemma \ref{lem-2}). Then we can apply Davenport-Erd\"os-LeVeque Theorem (Lemma \ref{thm-DEL}).

Let $f$ be a scaling map from $\mathcal{O}$ into $\mathcal{F}$.
Let ${\rm proj}(x)=[x]$ for $x\in \mathcal{F}$.
Then define the associated map
\[\tilde{f}={\rm proj} \circ f: \mathcal{O} \to \mathcal{O}.\]
Given a non-empty open set $\Omega$, we define a Borel probability measure on $\Omega$ by
\[\mu|_\Omega(B) = \frac{\mu(\Omega\cap B)}{\mu(\Omega)}.\]
It is the conditional measure of the Haar measure on $\Omega$.

Now we establish the invariance of the Haar measure under the map $\tilde{f}: \mathcal{O} \to \mathcal{O}$.
\begin{lemma}[$\tilde{f}$-invariance of Haar measure]\label{lem-1}
Assume  that  $f :\mathcal{O} \to \mathcal{F}$ is a scaling map having scaling ratio $q^\lambda$ with $\lambda\ge 0$.  
Then the Haar measure $\mu$ on $\mathcal{O}$ is $\tilde{f}$-invariant.
\end{lemma}
\begin{proof}
First notice that $f(\mathcal{O})$ is a disk of radius $q^{\lambda}$, which is the union of $q^\lambda$ disks of radius $1$.
We list these disks by $D_1, \cdots, D_{q^\lambda}$. Each disk $D_j$ is mapped onto $\mathcal{O}$ under the map ${\rm proj}(x) = [x]$.
This map is just a translation, the translation by the common fractional part $r_j$ of points in $D_j=D(r_j, 1)$. So Haar measure is preserved by the map ${\rm proj}:D_j \mapsto \mathcal{O}$.

\begin{figure}[ht]
 	\centering
 	\includegraphics[width=0.8\linewidth]{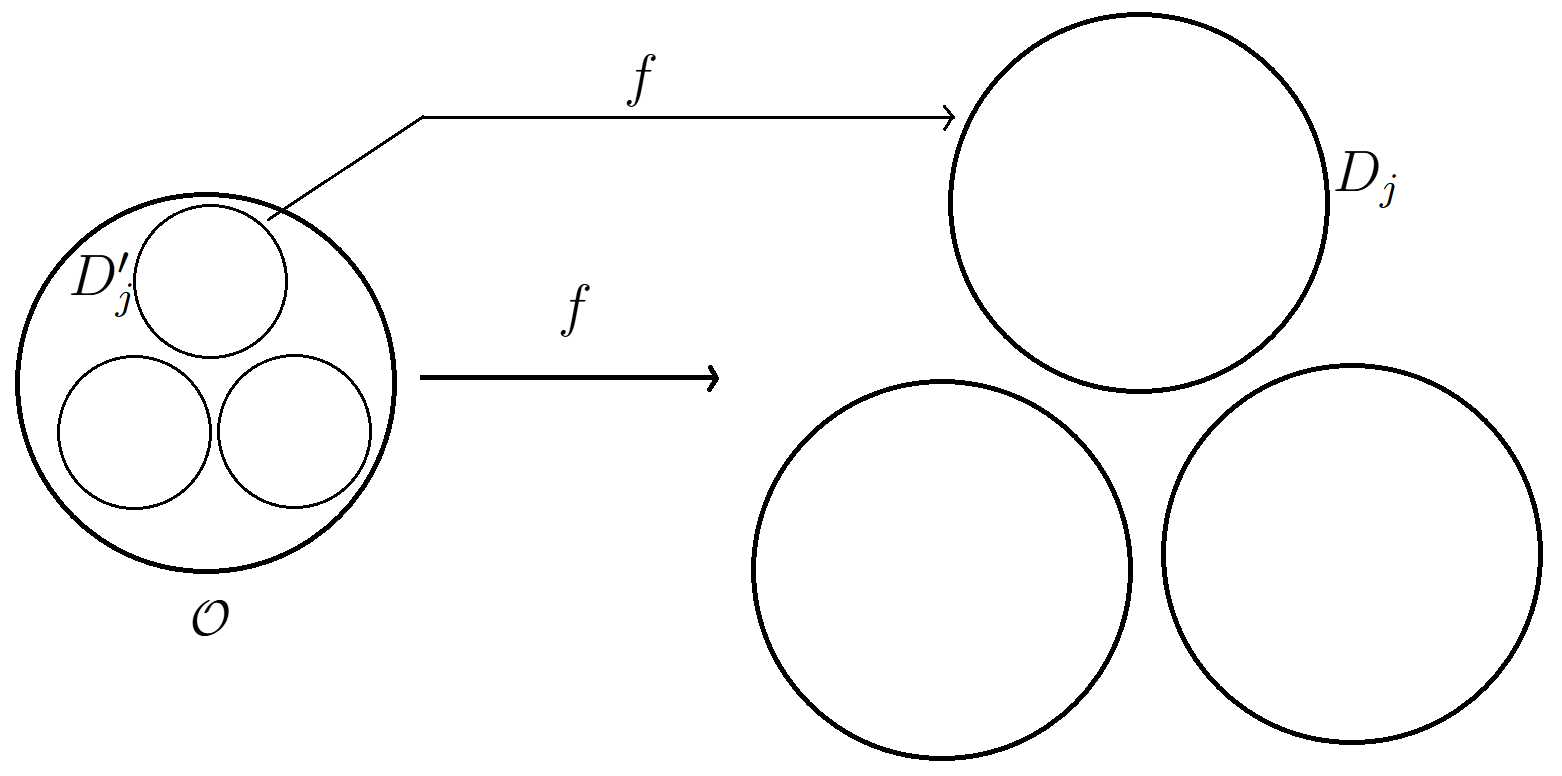}
 	\label{lem 3-1}
\end{figure}

Let $D_1' , \cdots, D_{p^{\lambda}}'$ be the preimages of $D_1, \cdots, D_{q^\lambda}$ under $f$, then we have $\mathcal{O} =\bigsqcup_{j=1} ^{q^\lambda}D_j'$ and the restriction $\tilde{f}: D_j' \to \mathcal{O}$ is scaling and bijective for each $j$.
As the Haar measure $\mu$ on $\mathcal{O}$ has the decomposition $\mu =q^{-\lambda}\sum_{j=1}^{q^\lambda} \mu|_{D_j'}$, we get  
\[\mu\circ \tilde{f}^{-1} = q^{-\lambda}\sum_{j=1}^{q^\lambda} \mu|_{D_j'}\circ \tilde{f}^{-1}.\]

Now we show $\mu|_{D_j'}\circ \tilde{f}^{-1}=\mu$ to finish the proof. Notice that $\tilde{f}: D_j' \to \mathcal{O}$ is a scaling bijection of scaling ratio $q^{\lambda}$ for each $j$. Take an arbitrary disk $B$ of radius $q^{-m}$ in $\mathcal{O}$. Then the preimage 
$\tilde{f}^{-1}(B)$ is a disk of radius $q^{-m-\lambda}$ contained in some $D_j'$. Therefore
\[\mu|_{D_j'}\circ \tilde{f}^{-1}(B) = \frac{\mu(\tilde{f}^{-1}(B))}{\mu(D_j')} = \frac{q^{-m-\lambda}}{q^{-\lambda}} = q^{-m} =\mu(B).\]
That means $\mu|_{D_j'}\circ \tilde{f}^{-1}=\mu$.
\end{proof}

The following lemma shows the non-correlation of $1_D \circ \tilde{f}$ and $1_D \circ \tilde{g}$ under suitable condition on the scaling functions $f$ and $g$ and on the disk $D$, where $1_B$ denotes the indicator function of a set $B$.
In the following, $\mathbb{E}$ refers to the expectation with respect to the Haar measure $\mu$.
 
\begin{lemma}\label{lem-2}
Let $f$ and $g$ be two scaling maps from $\mathcal{O}$ into $\mathcal{F}$ of scaling ratios $q^{\lambda_{f}}$ and $q^{\lambda_g}$
with $\lambda_f>\lambda_g\geq 0$.
Assume that $D\subset \mathcal{O}$ is a disk of radius $q^{-\gamma}$ with $0\le \gamma\le \lambda_f-\lambda_g$. Then 
$$
       \mathbb{E} (1_D \circ \tilde{f} \cdot 1_D \circ \tilde{g} ) =\mu(D)^2. 
$$
\end{lemma}
\begin{figure}[ht]
 	\centering
 	\includegraphics[width=0.8\linewidth]{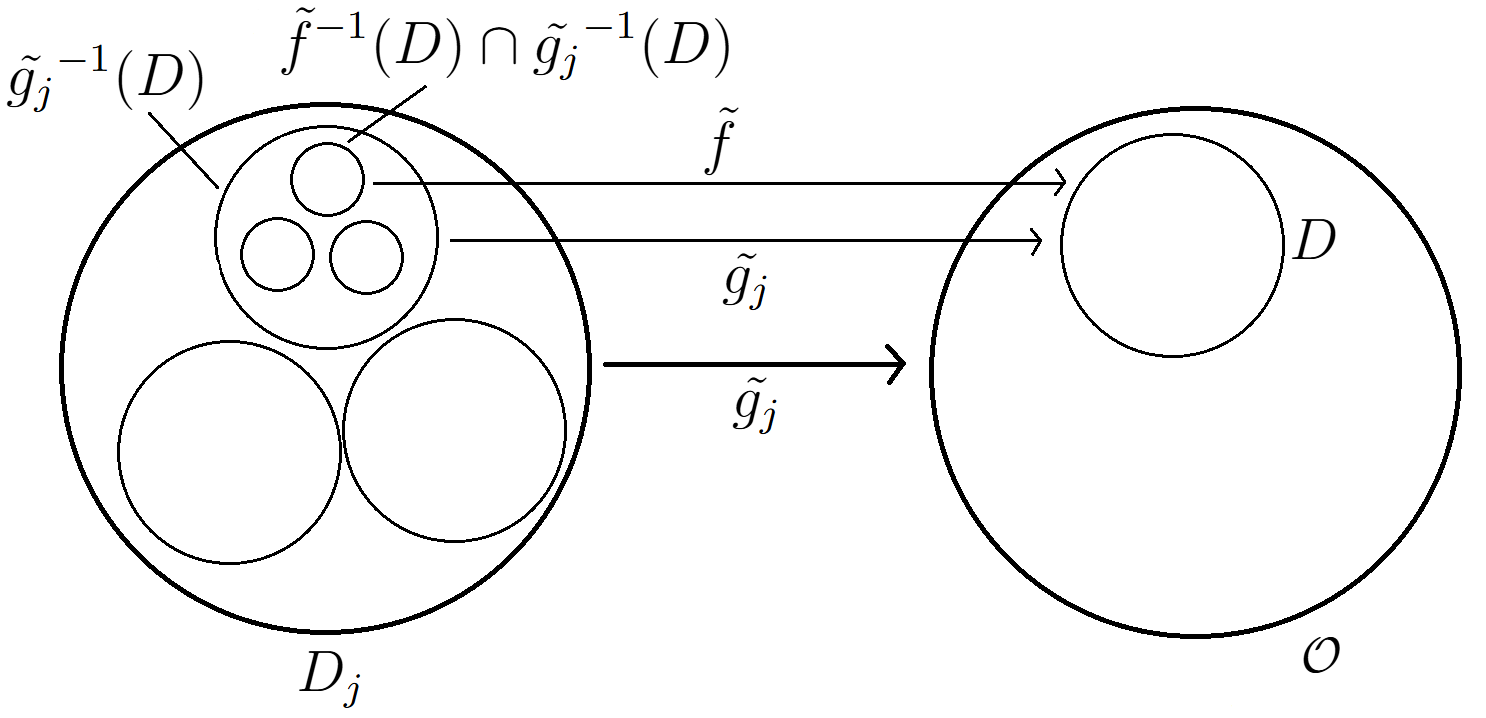}
 	\label{lem 3-2}
\end{figure}
\begin{proof}
Since $\mathbb{E} (1_D \circ \tilde{f} \cdot 1_D \circ \tilde{g} ) = \mu(\tilde{f}^{-1}(D)\cap \tilde{g}^{-1}(D))$, we only need to prove the following equality
\[ \mu(\tilde{f}^{-1}(D)\cap \tilde{g}^{-1}(D))= \mu(D)^2.\]
In the proof of Lemma \ref{lem-1}, we have seen that $\tilde{g}: \mathcal{O} \to \mathcal{O}$ is a $q^{\lambda_g}$-to-1 map. Decompose $\mathcal{O}$ into the union of the disks $D_j$ with $0\le j\le q^{\lambda_g}-1$, 
so that the restriction $\tilde{g_j}:= \tilde{g}: D_j \to \mathcal{O}$ is bijective. Therefore
$$
\mu(\tilde{f}^{-1}(D)\cap \tilde{g}^{-1}(D)) = \sum_{j=0}^{q^{\lambda_g}-1} \mu(\tilde{f}^{-1}(D)\cap \tilde{g_j}^{-1}(D)).
$$
Notice that for any fixed $j$, $\tilde{g_j}^{-1}(D)$ is a disk of radius $r_g:=q^{-\gamma -\lambda_g}$.

Similarly, each disk in $\mathcal{O}$ of radius $q^{-\lambda_f}$ is bijectively mapped by $\tilde{f}$ onto $\mathcal{O}$. So $\tilde{f}^{-1}(D)$
is a union of $q^{\lambda_f}$ disks of radius $r_f:= q^{-\gamma-\lambda_f}$, each of which is contained  in one disk of radius $q^{-\lambda_f}$.
Now notice that each $\tilde{g_j}^{-1}(D)$ is decomposed into $r_g/q^{-\lambda_f}$, namely $q^{\lambda_f -\lambda_g -\gamma}$, disks of radius $q^{-\lambda_f}$.
Thus
$$
   \mu(\tilde{f}^{-1}(D)\cap \tilde{g}^{-1}(D)) = q^{\lambda_g}\cdot q^{\lambda_f -\lambda_g -\gamma} \cdot q^{-\gamma-\lambda_f}=\mu(D)^2.
$$
\end{proof}

The following law of large number due to Davenport-Erd\"os-LeVeque will be needed.

\begin{lemma}[Davenport-Erd\"os-LeVeque 1963 \cite{DEL1963}]\label{thm-DEL}
Assume that $(Y_n)$ is a sequence of random variables defined on some probability space such that $\Vert Y_n\Vert_{\infty}=O(1)$.
Let
\[ X_{N}:=\frac{Y_0+\cdots+Y_{N-1}}{N}.\] 
We have $\lim_{N\to \infty}X_N=0$ almost surely, under the following condition
  \begin{equation}\label{eq:DEL}
  \sum_{N=1}^{\infty} \frac{\mathbb{E}(|X_{N}|^2)}{N}<\infty.
  \end{equation}
\end{lemma}

An elementary inequality will also be needed.

\begin{lemma}\label{lem-sum}
Let $n_1, n_2, n_3, \cdots, n_s$ be positive integers. For any $1\leq m \leq s$, we have 
\[ \sum _{\substack{1\leq i<j\leq s\\ j-i\leq m-1}}n_i n_j \leq  (m-1) \sum_{i=1}^{s}n_{i}^2.\]
\end{lemma}
\begin{proof}
The proof is elementary. We may assume that $1 < m \le s$ (the case $m=1$ being trivial). 
By the Cauchy inequality,
\[
\sum_{\substack{1 \le i < j \le s \\ j-i \le m-1}} n_i n_j 
\;\le\;
\frac{1}{2} 
\sum_{\substack{1 \le i < j \le s \\ j-i \le m-1}} (n_i^2 + n_j^2).
\]

Moreover,
\begin{align*}
\sum_{\substack{1 \le i < j \le s \\ j-i \le m-1}} (n_i^2 + n_j^2)
&= 
\sum_{i=1}^{s} 
\left( 
\sum_{\substack{i<j\le s \\ j \le i+m-1}} 1 
\right) n_i^2
\;+\;
\sum_{j=1}^{s} 
\left(
\sum_{\substack{1 \le i < j \\ i \ge j-m+1}} 1
\right) n_j^2 \\[0.4em]
&=
\sum_{i=1}^{s}
\bigl( 
\min\{s,\, i+m-1\} - i 
\;+\;
i - \max\{1,\, i-m+1\}
\bigr) n_i^2 \\[0.4em]
&\le
\sum_{i=1}^{s}
\bigl( (i+m-1) - (i-m+1) \bigr) n_i^2 \\[0.4em]
&=
\sum_{i=1}^{s} 2(m-1)\, n_i^2.
\end{align*}
\end{proof}

\begin{proof}[Proof of Theorem \ref{thm-main}]
Recall that $\Omega$ is assumed to be a disk in $\mathcal{F}$. Without loss of generality, we assume that $\Omega=\mathcal{O}$. Otherwise, take a bijective linear map $T:\mathcal{O} \to \Omega$.
That $(\tilde{f}_n(x))$ is uniformly distributed for almost all $x \in \Omega$ is equivalent to that $(\tilde{f}_n(T(x'))$ is uniformly distributed for almost all $x' \in \mathcal{O}$.

Let $D$ be an arbitrary disk in $\mathcal{O}$. Consider the centered random variables $$
Y_n(\omega):=1_{D}(\tilde{f}_n(\omega))-\mu(D).
$$
We have $\|Y_n\|_\infty \le 1$, and $\mathbb{E}Y_n=0$ by the $\tilde{f}$-invariance of the Haar measure (see Lemma \ref{lem-1}).
What we need to prove is that for any fixed disk $D$ in $\mathcal{O}$,
\begin{equation}\label{equ-average}
    \lim_{N\to \infty} X_N(\omega)=0
\end{equation}
for a.e. $\omega\in \mathcal{O}$, where 
$$
X_N=\frac{1}{N}\sum_{n=0}^{N-1}Y_n.
$$
We will prove it by Davenport-Erdos-LeVeque's law of large numbers, by checking the condition \eqref{eq:DEL}.
Note that there are only countably many disks in $\mathcal{O}$.
Hence, if for any disk $D$ in $\mathcal{O}$, \eqref{equ-average} is true for a.e. $\omega\in \mathcal{O}$, then for a.e. $\omega\in \mathcal{O}$ and any disk $D$ in $\mathcal{O}$, \eqref{equ-average} is true, which means the uniform distribution.

Fix the disk $D$. Assume that $\mu(D)=q^{-k}$. By Lemma \ref{lem-1} and Lemma \ref{lem-2}, we get
\begin{align}\label{eq-cor}
\mathbb{E}(Y_nY_m)=0, \hbox{ if } |\lambda_m-\lambda_n|\geq k.
\end{align}
Now we estimate the second moment
  \[
   \mathbb{E}(X_N^2)=\frac{1}{N^2}\sum_{0\leq m,n\leq N-1} \mathbb{E}(Y_m Y_n).
   \]
Let 
$$
   \Lambda_N:=\{\lambda_n: 0\leq n \leq N-1\}:=\{\gamma_1, \gamma_2, \cdots, \gamma_s\}.
$$ 
Consider the partition
$$
\{1,2,\cdots, N\} = \bigcup_{j=1}^s A_{N,j} \ \ {\rm with} \ \ A_{N,j} =\{0\le n\le N-1: \lambda_n =\gamma_j\}.
$$
Then
$$
    \mathbb{E}X_N^2 = \frac{1}{N^2}\sum_{i=1}^s\sum_{j=1}^s  \mathbb{E}\left(\sum_{\lambda_n \in A_{N,i}} Y_n\right) \left(\sum_{\lambda_m \in A_{N,j}} Y_m\right)
$$
For $1\leq j\leq s$, let $n_j=\#A_{N,j}$.  
By \eqref{eq-cor} and $\|Y_n\|_\infty \le 1$, we get
\begin{equation}\label{eq:1}
\mathbb{E}(X_N^2) \leq \frac{1}{N^2} \sum_{|\gamma_i-\gamma_j|< k}n_in_j.
\end{equation}
We may assume that $\gamma_1<\gamma_2<\cdots<\gamma_s$. Then by Lemma \ref{lem-sum},
\begin{equation}\label{eq:2}
\sum_{|\gamma_i-\gamma_j|< k}n_in_j
\le \sum_{i=1}^{s}n_{i}^2 + 2\sum_{\substack{1 \le i < j \le s \\ j-i \le k-1} }n_in_j
\le (2k-1) \sum_{i=1}^{s}n_{i}^2.
\end{equation}
 But $\sum_{i=1}^{s}n_{i}^2=|K_N|$. This, together with \eqref{eq:1} and \eqref{eq:2}, implies the condition \eqref{eq:DEL}.
\end{proof}

\section{Proof of Theorem \ref{thm-koksma} and Theorem \ref{thm-char=p}}

In order to prove Theorem \ref{thm-koksma}, we show that the maps $f_n(x)=\alpha x^n$ are locally scaling (see Lemma \ref{lem-scaling}) and the scaling ratios are  ``nearly" pairwise different (see Lemma \ref{lem-logN}). Then we use Theorem \ref{thm-main} to finish the proof.

We use similar methods to prove Theorem \ref{thm-char=p} by considering scaling maps on $\widetilde{\mathcal{S}_v}$ instead of scaling maps on $\mathcal{O}$.

\begin{lemma}\label{lem-scaling}
Let $a\in \mathcal{F}\setminus\{0\}$, $\alpha\in \mathcal{F}\setminus\{0\}$, and let $n$ be a positive integer.
\begin{itemize}
\item[{\rm(a)}] If ${\rm char}\,\mathcal{F}=0$, then  the map $f_n(x)=\alpha x^{n}$ satisfies
\[
|f(x)-f(y)|_\mathfrak{p} = |\alpha|_\mathfrak{p}\, |n|_\mathfrak{p}\, |a|_\mathfrak{p}^{\,n-1}\,|x-y|_\mathfrak{p}
\]
on the disk 
\[
D\!\left(a,\, |a|_\mathfrak{p}/q^{\,e+1}\right),
\]
where $e$ denotes the ramification index of the extension $\mathcal{F}/\Q_p$.

\item [{\rm(b)}] If ${\rm char}\,\mathcal{F}=p>0$, then the map $f_n(x)=\alpha x^{n}$ satisfies
\[
|f(x)-f(y)|_\mathfrak{p} = |\alpha|_\mathfrak{p}\,|a|_\mathfrak{p}^{\,n-p^{v_p(n)}}\,|x-y|_\mathfrak{p}^{\,p^{v_p(n)}}
\]
for all $x,y$ in the disk $D(a, |a|_\mathfrak{p}/q)$.  
In particular, if $p\nmid n$, then $f_n(x)=\alpha x^{n}$ is scaling on the disk $D(a, |a|_\mathfrak{p}/q)$ with scaling ratio 
\[
|\alpha|_\mathfrak{p}\,|a|_\mathfrak{p}^{\,n-1}.
\]
\end{itemize}
\end{lemma}
\begin{proof}
\textbf{(a)  ${\rm char}\mathcal{F}=0$.}
For $x,y\in D(a, |a|_\mathfrak{p}/q^{\,e+1})$, we have
\[
|x-y|_\mathfrak{p}
   \leq \max\{|x-a|_\mathfrak{p},\,|y-a|_\mathfrak{p}\}
   \leq |a|_\mathfrak{p}/q^{\,e+1},
   \qquad 
|x|_\mathfrak{p}=|y|_\mathfrak{p}=|a|_\mathfrak{p}.
\]
By direct computation, we have  
\begin{align*}
|f(x)-f(y)|_\mathfrak{p}
 &= |\alpha|_\mathfrak{p}\,|x^{n}-y^{n}|_\mathfrak{p} \\
 &= |\alpha|_\mathfrak{p}
    \left| n y^{\,n-1}(x-y)
        + \sum_{i=2}^n \binom{n}{i} y^{\,n-i}(x-y)^i
    \right|_\mathfrak{p} 
  \end{align*}

For $2\le i\le n$, using $v_\mathfrak{p}(k)=e\,v_p(k)$ for any integer $k$, we compute
\begin{align*}
v_\mathfrak{p}\left(\!\binom{n}{i}\!\right)
&= e\,v_p\left(\!\binom{n}{i}\!\right)
 = e\,v_p(n) - e\,v_p(i) + e\,v_p\left(\!\binom{n-1}{i-1}\!\right) \\
&\ge e\,v_p(n) - e\,v_p(i)
 \ge e\,v_p(n) - e(i-1) \\
&= v_\mathfrak{p}(n)-e(i-1)  
  > v_\mathfrak{p}(n)-(e+1)(i-1).
\end{align*}
Thus
\[
\frac{|n y^{\,n-1}(x-y)|_\mathfrak{p}}
     {\bigl|\binom{n}{i} y^{\,n-i}(x-y)^i\bigr|_\mathfrak{p}}
   = \frac{|n|_\mathfrak{p}}{\bigl|\binom{n}{i}\bigr|_\mathfrak{p}}
     \left|\frac{y}{x-y}\right|_\mathfrak{p}^{\,i-1}
   >1.
\]
By the strong triangle inequality, we have
\begin{align*}
|f(x)-f(y)|_\mathfrak{p}
&= |\alpha|_\mathfrak{p}\,|n y^{\,n-1}(x-y)|_\mathfrak{p} \\
 &= |\alpha|_\mathfrak{p}\,|n|_\mathfrak{p}\,|a|_\mathfrak{p}^{\,n-1}\,|x-y|_\mathfrak{p}.
\end{align*}

\medskip
\textbf{(b)  ${\rm char}\mathcal{F}=p$.}
For $x,y\in D(a, |a|_\mathfrak{p}/q)$, we again have
\[
|x-y|_\mathfrak{p}
   \le |a|_\mathfrak{p}/q,
   \qquad 
|x|_\mathfrak{p}=|y|_\mathfrak{p}=|a|_\mathfrak{p}.
\]

\emph{Case 1: $p\nmid n$.}  
Since $|\binom{n}{i}|_\mathfrak{p}\in\{0,1\}$ for $i\ge2$, the same argument as in (a) gives
\begin{align*}
|f(x)-f(y)|_\mathfrak{p}
 &=|\alpha|_\mathfrak{p}\,|x^{n}-y^{n}|_\mathfrak{p} \\
 &=|\alpha|_\mathfrak{p}\,|n y^{\,n-1}(x-y)|_\mathfrak{p} \\
 &=|\alpha|_\mathfrak{p}\,|a|_\mathfrak{p}^{\,n-1}\,|x-y|_\mathfrak{p}.
\end{align*}

\emph{Case 2: $p\mid n$.}  
Write $n=n'p^{v}$ with $p\nmid n'$. Then
\begin{align*}
|f(x)-f(y)|_\mathfrak{p}
 &=|\alpha|_\mathfrak{p}\,|x^{n'p^{v}}-y^{n'p^{v}}|_\mathfrak{p} \\
 &=|\alpha|_\mathfrak{p}\,|(x^{n'}-y^{n'})^{p^{v}}|_\mathfrak{p} \\
 &=|\alpha|_\mathfrak{p}\,|a|_\mathfrak{p}^{\,n-p^{v}}\,|x-y|_\mathfrak{p}^{\,p^{v}}.
\end{align*}
\end{proof}

Now we assume ${\rm char}\mathcal{F}=0$.
For positive integers $k, n, N,\gamma$ with $k\leq n\leq N+k-1$, let
\[\Lambda^{\gamma}_{n,k,N}=\{k\leq m\leq N+k-1: m\gamma-v_\mathfrak{p}(m)=n\gamma-v_\mathfrak{p}(n) \}.\]
Now, we provide a crucial estimate for the cardinality of $\Lambda^{\gamma}_{n,k, N}$.
\begin{lemma}\label{lem-logN}
For positive integers $k, n, N,\gamma$ with $k\leq n\leq N+k-1$, we have
\[\#\Lambda^{\gamma}_{n, k, N}\leq \frac{e \log_p (N+k)}{\gamma} +1.\]
\end{lemma}
\begin{proof}
Assume that 
\[
\Lambda^{\gamma}_{n,k,N}=\{m_1, m_2, \ldots, m_s\}
\quad\text{with}\quad
m_1 < m_2 < \cdots < m_s.
\]
For any $m \in \Lambda^{\gamma}_{n,k,N}$, we have
\[
v_{\mathfrak{p}}(m)
= v_{\mathfrak{p}}(n) + (m-n)\gamma.
\]
Hence,
\[
v_{\mathfrak{p}}(m_1)
< v_{\mathfrak{p}}(m_2)
< \cdots <
v_{\mathfrak{p}}(m_s),
\qquad
|v_{\mathfrak{p}}(m_i)-v_{\mathfrak{p}}(m_j)| \ge \gamma
\quad (i \ne j).
\]

Since 
\[
0 \le v_{\mathfrak{p}}(m_j) = e\, v_{p}(m_j)
< e \log_p (N+k),
\]
the values $v_{\mathfrak{p}}(m_j)$ lie in an interval of length $e\log_p(N+k)$.  
Because they are strictly increasing with pairwise gaps at least $\gamma$, we obtain
\[
s \le \frac{e\log_p(N+k)}{\gamma} + 1.
\]
\end{proof}

\begin{proof}[Proof of Theorem \ref{thm-koksma}]
Fix any $a\in \mathcal{F}$ with $|a|_\mathfrak{p}>1$ and set $f_n(x)=\alpha x^n$.

\textbf{(a)  ${\rm char}\mathcal{F}=0$.}  
Consider the disk
\[
\Omega := D(a, |a|_\mathfrak{p}/q^{\,e+1}).
\]
Let $d = e \log_p q$ and choose a positive integer $n_0$ such that for all $n\ge n_0$, 
\[
|\alpha|_\mathfrak{p} n^{-d} q^{\,n-(e+1)} \ge 1.
\]
By Lemma \ref{lem-scaling}(a), the map $f_n$ is scaling on $\Omega$ with scaling ratio 
\[
q^{\lambda_n}:=|\alpha|_\mathfrak{p} |n|_\mathfrak{p} |a|_\mathfrak{p}^{\,n-1}.
\]
For $n \ge n_0$,
\[
|\alpha|_\mathfrak{p} |n|_\mathfrak{p} |a|_\mathfrak{p}^{\,n-1} \, {\rm diam}(\Omega)
= |\alpha|_\mathfrak{p} q^{-(e+1)} |n|_\mathfrak{p} |a|_\mathfrak{p}^{\,n}
\ge |\alpha|_\mathfrak{p} n^{-d} q^{\,n-(e+1)} \ge 1.
\]

By Lemma \ref{lem-logN}, for $n_0 \le n \le N+n_0-1$,
\[
\#\Bigl\{ m : n_0 \le m \le N+n_0-1, \ |m|_\mathfrak{p} |a|_\mathfrak{p}^{\,m-1} = |n|_\mathfrak{p} |a|_\mathfrak{p}^{\,n-1} \Bigr\}
\le \frac{e \log_p (N+n_0)}{\log_q |a|_\mathfrak{p}} + 1.
\]
Hence,
\[
\#\Bigl\{(n,m): n_0 \le n,m \le N+n_0-1, \ \lambda_n = \lambda_m \Bigr\}
\le N \left( \frac{e \log_p (N+n_0)}{\log_q |a|_\mathfrak{p}} + 1 \right).
\]

Applying Theorem \ref{thm-main}, the sequence $([\alpha x^n])_{n\ge n_0}$ is uniformly distributed in $\mathcal{O}$ for almost all $x\in \Omega$, which implies that $([\alpha x^n])_{n\ge 1}$ is also uniformly distributed in $\mathcal{O}$.  
Since the set $\{ x\in\mathcal{F} : |x|_\mathfrak{p}>1 \}$ is covered by countably many such disks $\Omega$, it follows that $([\alpha x^n])$ is uniformly distributed in $\mathcal{O}$ for almost all $x$ with $|x|_\mathfrak{p}>1$.

\medskip
\textbf{(b)  ${\rm char}\mathcal{F}=p$.}  
By Lemma \ref{lem-scaling}(b), the subsequence $(f_n)_{p \nmid n}$ consists of locally scaling maps on the disk 
\[
\Omega := D(a, |a|_\mathfrak{p}/q)
\]
with strictly increasing scaling ratios $|\alpha|_\mathfrak{p} |a|_\mathfrak{p}^{\,n-1}$.  

By Theorem \ref{thm-main}, the sequence $([\alpha x^n])_{n\ge 1,\,p\nmid n}$ is uniformly distributed in $\mathcal{O}$ for almost all $x\in \Omega$.  
Since $\{ x\in\mathcal{F} : |x|_\mathfrak{p}>1 \}$ is covered by countably many such disks, the result follows for almost all $x$ with $|x|_\mathfrak{p}>1$.
\end{proof}

In order to proof Theorem \ref{thm-char=p}, we split the sequence $([x^n])_{n\geq 1}$ into some subsequences based on $v_p(n)$. For $k\geq 1$, we consider the distribution of the subsequence $([x^n])_{p^k\parallel n}$ in $\widetilde{\mathcal{S}_k}$.

\begin{proof}[Proof of Theorem \ref{thm-char=p}]
\textbf{(a)}
Suppose first that $p^k \mid n$. For any $x=\sum_{i=v_\mathfrak{p}(x)}^\infty x_i\pi^i$, we have
\[
x^n=\left(\sum_{i=v_\mathfrak{p}(x)}^\infty x_i\pi^{i}\right)^{p^k\cdot n/p^k}
   =\left(\sum_{i=v_\mathfrak{p}(x)}^\infty x_i^{p^k}\pi^{ip^k}\right)^{n/p^k}
   \in \mathcal{S}_k.
\]
Thus $[x^n]\in\widetilde{\mathcal{S}_k}$ for all $x\in\mathcal{F}$ whenever $p^k\mid n$.

In particular, for $p\mid n$ and any $x\in\mathcal{F}$ we have $[x^n]\in\widetilde{\mathcal{S}_1}$.  
Since $\widetilde{\mathcal{S}_1}$ is compact of Haar measure $0$ in $\mathcal{O}$, given $0<\epsilon<1/p$ there exist finitely many open disks $\{D_j\}_{1\le j\le J}$ covering $\widetilde{\mathcal{S}_1}$ such that $\sum_{j=1}^J \mu(D_j)<\epsilon$.
For any fixed $x$ with $|x|_\mathfrak{p}>1$,
\begin{align*}
\lim_{N\to\infty} \frac{1}{N}\#\{1\le n\le N:[x^n]\in \cup_j D_j\}
&\ge \lim_{N\to\infty} \frac{1}{N}\#\{1\le n\le N:p\nmid n\}  \\
&=1/p>\epsilon > \mu\left(\bigcup_j D_j\right).
\end{align*}
Hence $([x^n])_{n\ge1}$ is not Haar-u.d.\ in $\mathcal{O}$ for any $x$ with $|x|_\mathfrak{p}>1$.

\textbf{(b)}
Define
\[
g_k:\mathcal{F}\to \mathcal{S}_k,\qquad g_k(x)=x^{p^k}.
\]
We claim that $g_k$ is an isomorphism of measure spaces $(\mathcal{F},\mu)$ and $(\mathcal{S}_k,\mu_k)$.

Since $x\mapsto x^{p^k}$ is an automorphism of $\mathbb{F}_q$, $g_k$ is bijective.  
For any ball $b+\pi^l\mathcal{O}\subset\mathcal{F}$,
\[
g_k(b+\pi^l\mathcal{O})=b^{p^k}+\pi^{lp^k}\widetilde{\mathcal{S}_k},
\]
and
\[
\mu(b+\pi^l\mathcal{O})=q^{-l}
   =\mu_k\big(b^{p^k}+\pi^{lp^k}\widetilde{\mathcal{S}_k}\big).
\]
Thus the claim holds.

If $p\nmid n'$, then $f_{n'}(x)=x^{n'}$ is locally scaling on the disk
\[
\Omega=D(a,|a|_\mathfrak{p}/q^{p^k}) \subset \mathcal{S}_k
\quad (0\neq a\in\mathcal{S}_k)
\]
with scaling ratio $|a|_\mathfrak{p}^{n'-1}$ by Lemma \ref{lem-scaling}(b).  
Exactly as in Theorems \ref{thm-main} and \ref{thm-koksma}(b), one obtains that
\[
([x^{n'}])_{n'\ge1,\; p\nmid n'} 
\quad\text{is $\mu_k$-u.d.\ in }\widetilde{\mathcal{S}_k}
\]
for $\mu_k$-a.e.\ $x\in \mathcal{S}_k$ with $|x|_\mathfrak{p}>1$.
Since $x^n=(g_k(x))^{\,n/p^k}$ whenever $p^k\parallel n$, the claim implies that
\[
([x^n])_{n\ge1,\; p^k\parallel n}
\quad\text{is $\mu_k$-u.d.\ in }\widetilde{\mathcal{S}_k}
\]
for $\mu$-a.e.\ $x\in\mathcal{F}$ with $|x|_\mathfrak{p}>1$.

\textbf{(c)}
By Theorem \ref{thm-koksma}(b) and part (b) above, for a.e.\ $x$ with $|x|_\mathfrak{p}>1$ the following hold:
\begin{equation}\label{eq:thm1.4(1)}
\frac{1}{N(1-1/p)}\!\!\sum_{\substack{1\le n\le N\\ p\nmid n}}\!\delta_{[x^n]}
   \xrightarrow[N\to\infty]{{\rm w}^*}\mu,
\quad
\frac{1}{N(1-1/p)p^{-k}}\!\!\sum_{\substack{1\le n\le N\\ p^k\parallel n}}\!\delta_{[x^n]}
   \xrightarrow[N\to\infty]{{\rm w}^*}\mu_k, \ \forall k\geq 1.
\end{equation}
(Here $\mu_k$ is viewed as a measure on $\mathcal{O}$ via $\mu_k(D)=\mu_k(D\cap\widetilde{\mathcal{S}_k})$.)

We must show that for a.e.\ such $x$,
\begin{equation}\label{eq:thm1.4(3)}
\frac{1}{N}\sum_{1\le n\le N}\delta_{[x^n]}
   \xrightarrow[N\to\infty]{{\rm w}^*}
   \left(1-\frac{1}{p}\right)\!\left(\mu+\sum_{k=1}^\infty p^{-k}\mu_k\right).
\end{equation}
Since
\[
\frac{1}{N}\sum_{1\le n\le N}\delta_{[x^n]}
   =\frac{1}{N}\sum_{k=0}^\infty
      \sum_{\substack{1\le n\le N\\ p^k\parallel n}}\!\delta_{[x^n]},
\]
it follows from \eqref{eq:thm1.4(1)} that for any fixed $K$,
\[
\frac{1}{N}\sum_{k=0}^{K}
   \sum_{\substack{1\le n\le N\\ p^k\parallel n}}\!\delta_{[x^n]}
   \xrightarrow[N\to\infty]{{\rm w}^*}
   \left(1-\frac{1}{p}\right)\!\left(\mu+\sum_{k=1}^K p^{-k}\mu_k\right).
\]
Meanwhile,
\[
\left\|\frac{1}{N}\sum_{k=K+1}^{\infty}
   \sum_{\substack{1\le n\le N\\ p^k\parallel n}}
   \delta_{[x^n]}\right\|
   \le \frac{1}{N}\cdot \frac{N}{p^{K+1}}
   =p^{-(K+1)} \to 0
   \quad(K\to\infty).
\]
Thus \eqref{eq:thm1.4(3)} follows.
\end{proof}

\begin{remark}
If we consider $[\alpha x^n]$ for an arbitrary $\alpha$, then the question of whether the sequence $\bigl([\alpha x^n]\bigr)$ is u.d.\ in $\mathcal{O}$ remains open.
For instance, define
\[
\mathcal{S}' := \{[\alpha y] : y \in \mathcal{S}_1\}.
\]
The set $\mathcal{S}'$ may even be dense in $\mathcal{O}$.
If $\alpha$ is ``good'' in the sense that the measure of the closure $\overline{\mathcal{S}'}$ is sufficiently small, then the sequence $\bigl([\alpha x^n]\bigr)$ fails to be u.d.\ in $\mathcal{O}$ for every $x$.
\end{remark}

\section{The exceptional set of \texorpdfstring{$([f_n(x)])$}. of increasing scaling ratios}

The idea of the proof for Theorem \ref{thm-fulldimscaling} is that the exceptional set of points $x$, for which the first non-Archimedean digit of $[f_n(x)]$ fails to be uniformly distributed has a nice structure and  
has full Hausdorff dimension under fairly general assumptions. 
To prove Theorem \ref{thm-fulldimscaling}, we begin with the following lemma, which is a variant of the main theorem in \cite{Egg1949}.

Let $\Sigma_m^{\N}=\{0,1,\ldots,m-1\}^{\N}$ be a symbolic space with metric $d(x,y)=m^{-\min\{i:x_i \neq y_i\}}$. Let $\Lambda=(\lambda_n)$ be a sequence of strictly increasing positive integers.
The upper density of $\Lambda$ in $\N$ is
\[\rho:= \limsup_{n\rightarrow\infty} \frac{\#(\Lambda \cap [1, n])}{n}
= \limsup_{n\rightarrow\infty} \frac{n}{\lambda_n}.\]
Let $P=(p_0,p_1,\ldots,p_{m-1})$ be a given probability vector. We define the set of points having $P$ as frequency vector along $\Lambda$ by
\[F_\Lambda(P):= \left\{x\in\Sigma_m^{\N}: \lim_{n\rightarrow\infty} \frac{\#\{1\leq i\leq n: x_{\lambda_i}=j\}}{n}=p_j \hbox{ for } j=0,1,\ldots,m-1\right\}.\]
Notice that  $x_n$'s are free for $n \not \in \Lambda$ if $x\in F_\Lambda$.

\begin{lemma}\label{lem-nnormal}
Under the above assumption, we have
\begin{equation}\label{eq:Hausdorff}
\dim_{\mathcal{H}}F_\Lambda(P)=(1-\rho) -\rho \sum_{j=0}^{m-1}p_j\log_m p_j.
\end{equation}In particular,
\[\dim_{\mathcal{H}} \{x\in\Sigma_m^{\N}: (x_{\lambda_n})_{n\geq1} \hbox{ not u.d. in } \Sigma_m\}=1,\]
where $\Sigma_m$ is equipped with the uniform probability measure.
\end{lemma}

\begin{proof}
The proof is nearly the same as the proof of proposition 10.1 in \cite{Fal2014}.
We define a Bernoulli measure $\mu_P$ on $\Sigma_m^{\N}$, which is the infinite product measure such that
\[\mu_P(\{x\in\Sigma_m^{\N}:x_k=j\})= \left\{\begin{aligned}
    &  p_j \hbox{, \ \  if } k\in\Lambda
    \\
    &\frac{1}{m} \hbox{, \ \ if } k \notin \Lambda .
\end{aligned}\right.\]
Let $\ell_n =\#\Lambda_n$ with $\Lambda_n=\Lambda \cap[0, n-1]$.
Denote the $n$-cylinder by $D_n(y)$.
Then $\limsup_{n\rightarrow\infty}\frac{\ell_n}{n}=\rho$ and
\begin{equation}\label{eq:local}
\mu_P(D_n(y)) = m^{-(n-\ell_n)}\prod_{j=0}^{m-1} p_j^{\#\{k\in \Lambda_n: y_k=j\}}.
\end{equation}

By SLLN, $\mu_P(F_\Lambda(P))=1$.
On the other hand, for $y\in F_\Lambda(P)$, from \eqref{eq:local} we get
\[\liminf_{n\rightarrow\infty}\frac{\log_m\mu_P(D_n(y))}{-n} =
1-\rho-\rho\Sigma_{j=0}^{m-1}p_j\log_m p_j.\]

The above two facts allow us to apply the mass distribution principle to confirm the formula \eqref{eq:Hausdorff}.
\medskip

If $P=(p_0,p_1,\ldots,p_{m-1})\not= (1/m,\ldots,1/m)$, $F_\Lambda(P)$ is a subset of the set
\[\{x\in\Sigma_m^{\N}: (x_{\lambda_n})_{n\geq1} \hbox{ not u.d. in } \Sigma_m\}.\]
But we can take 
$P=(p_0,p_1,\ldots,p_{m-1})$ tending to $(1/m,\ldots,1/m)$
so that $\dim_{\mathcal{H}} F_\Lambda(P)$ tends to $1$, according to \eqref{eq:Hausdorff}.  Therefore, we conclude that
\[\dim_{\mathcal{H}} \{x\in\Sigma_m^{\N}: (x_{\lambda_n})_{n\geq1} \hbox{ not u.d. in } \Sigma_m\}=1.\]
\end{proof}

\begin{proof}[Proof of Theorem \ref{thm-fulldimscaling}]
Without loss of generality, we may assume that $\Omega=\mathcal{O}$ and $\lambda_n\geq0$.

To prove this theorem, we construct an isometric bijection on $\mathcal{O}$. 
For a sequence of scaling maps $g_{n}:\mathcal{O} \to \mathcal{F}$  of scaling ratios $q^n$, we define a map \[\psi_{\{g_n\}}:\mathcal{O} \to \mathcal{O}\]
\[x\mapsto [g_0(x)]_0+[g_1(x)]_0 \pi+\cdots+[g_n(x)]_0 \pi^n+\cdots,\]
where $[g_n(x)]_0$ is the 0-digit of the non-Archimedean integer $[g_n(x)]$.

Now we claim that for $x=\sum_{j=0}x_j \pi^j,\ y=\sum_{j=0}y_j \pi^j \in \mathcal{O}$, if \( x_0 = y_0, \ldots, x_{k-1} = y_{k-1} \), then we have
\[[{g_k(x)}]_0=[{g_k(y)}]_0 \Leftrightarrow x_k=y_k.\]
Thus the map $\psi_{\{g_{n}\}}$ is an isometry.

Since $g_k:\mathcal{O} \to \mathcal{F}$ is a scaling map of scaling ratio $q^k$, then for each disk $D \subset \mathcal{O}$  of radius $q^{-k}$, the map $[g_k]:D\to \mathcal{O} \quad x\mapsto [g_k(x)]$ is a bijective map.
Hence,
\[\forall x,y \in D, \quad |[g_k(x)]-[g_k(y)]|_\mathfrak{p}=q^k|x-y|_\mathfrak{p}.\]
For $x=\sum_{j=0}x_j \pi^{j}, y=\sum_{j=0}y_j \pi^{j} \in \mathcal{O}$,  if \( x_0 = y_0, \ldots, x_{k-1} = y_{k-1} \), then $|x-y|_\mathfrak{p}\leq q^{-k}$.
It follows that
\[|[g_k(x)]-[g_k(y)]|_\mathfrak{p}=q^k|x-y|_\mathfrak{p} \leq 1,\]
which implies that
\[[g_k(x)]_0 = [g_k(y)]_0 \quad \hbox{if and only if} \quad x_k=y_k.\]
The claim is proved.

Choose a sequence of scaling maps $g_{n}:\mathcal{O} \to \mathcal{F}$  of scaling ratios $q^n$ such that $g_{\lambda_n}=f_n$. Let 
\[E=\left\{x=\sum_{n=0}^{\infty}x_n \pi^n\in \mathcal{O}: (x_{\lambda_n}) \hbox{\ is not u.d. in } R \right \},\]
where $R\subset \mathcal{O}$ is a system of representatives for $\mathcal{O}/\mathfrak{p}$ with $0\in R$.
By Lemma \ref{lem-nnormal}, $\dim_{\mathcal{H}}(E)=1$.
By the definitions of $\psi_{\{g_n\}}$ and $E$,
\[\psi_{\{g_n\}}^{-1}(E) = \{ x\in \mathcal{O}: ([f_n(x)]_0) \hbox{ is not u.d. in } R\}.\]
Since $\psi_{\{g_n\}}:\mathcal{O} \to \mathcal{O}$ is an isometric bijection, we have \[\dim_{\mathcal{H}} \psi_{\{g_n\}}^{-1}(E)=\dim_{\mathcal{H}}E=1.\]
\end{proof}

\section{The exceptional set of \texorpdfstring{$([\alpha x^n])$}.}

If ${\rm char}\mathcal{F}=0$, the map $f_n(x)=\alpha x^n$ is local scaling of local scaling ratio $|\alpha|_\mathfrak{p}|n|_\mathfrak{p}|x|_\mathfrak{p}^{n-1}$ in some neighborhood of $x$.
Unfortunately we can not find a subsequence of density 1 of $([\alpha x^n])$ with increasing scaling ratios, so we can not derive Theorem \ref{cor-fulldim2} directly from Theorem \ref{thm-fulldimscaling}. 
Theorem \ref{cor-fulldim2} is an immediate corollary of the following proposition.
It is a non-Archimedean version of Proposition 3 in \cite{Kahane2014} showing that the set of highly biased points is of full Hausdorff dimension.

\begin{proposition}\label{prop-fulldim2}
Suppose ${\rm char}\mathcal{F}=0$. Let $K\in \N_+$. For any sequence $(b_n)$ in $\mathcal{O}$ and positive integer $H$, take
\[E:=\{x\in \mathcal{F}: |x|_\mathfrak{p}>1;\; |[\alpha x^n]-b_n|_\mathfrak{p}\leq q^{-H},\; \forall n>0 \hbox{ with } p^K\nmid n\},\]
then $\dim_{\mathcal{H}}E=1$.
\end{proposition}

In order to prove Prop \ref{prop-fulldim2}, we show that the set $E$ has a $q$-homogeneous structure by the following Lemma \ref{thm-main-ebp}, which is a non-Archimedean version of Theorem 7 in \cite{BLM2019}.
Moreover, we need Lemma \ref{lem-dimlub}, which is a simple application of Example 4.6 in \cite{Fal2014} and Proposition 2.3 in \cite{Fal1997}, to explicitly compute the Hausdorff dimension.

\begin{lemma}\label{lem-dimlub}
Let $E_0=\mathcal{O} \supset E_1 \supset E_2 \supset \cdots$, where $E_k$ is the union of finitely many disks (called $k$-level disks). Suppose that each $(k-1)$-level disk contains at least $m_k \geq 2$ $k$-level disks. Let $E=\bigcap_{k\geq 0}{E_k}$.
\begin{itemize}
\item[(a)] If the distance between two $k$-level disks is at least $\delta_k$ with $\delta_k \to 0$, then
\[\dim_{\mathcal{H}}E \geq \liminf_{k\to \infty} \frac{\log (m_1\cdots m_{k-1})}{-\log (m_k\delta_k)}.\]
\item[(b)] If each $(k-1)$-level disk contains $m_k \geq 2$ $k$-level disks and the radius of $k$-level disks is at most $d_k$ with $d_k \to 0$, then
\[\dim_{\mathcal{H}}E \leq \liminf_{k\to \infty} \frac{\log (m_1\cdots m_k)}{-\log (d_k)}.\]
\end{itemize}
\end{lemma}

\begin{lemma}\label{thm-main-ebp}
Let $\Omega$ be a disk in $\mathcal{F}$ of radius $q^{-H_0}$ and $(H_n)_{n\geq 1}$ be a sequence of integers. Let $(f_{n})$ be a sequence of scaling maps defined on $\Omega$ of scaling ratios $(q^{\lambda_n})$. Suppose 
\[
\lambda_1\geq H_0
\quad \text{and} \quad 
\lambda_{n+1}-\lambda_n\geq H_{n}.
\]
Then, for any sequence $(b_n)$ in $\mathcal{O}$, the set 
\[
\Gamma_{b, H}:= \{x\in \Omega : |[f_n(x)]-b_n|_\mathfrak{p}\leq q^{-H_n},\; \forall n\ge 1\}
\]
is a $q$-homogeneous subset of $\Omega$ and admits its Hausdorff dimension
\begin{equation}\label{dim-of-hom}
\dim_{\mathcal{H}}\Gamma_{b, H}
= \liminf_{n\to\infty} 
\frac{\lambda_n-\sum_{k=1}^{n-1}H_k}
{\lambda_n+H_n}.
\end{equation}
\end{lemma}

The assumption $\lambda_{n+1}-\lambda_n\geq H_{n}$ is a lacunary condition.
Note that although the set $\Gamma_{b, H}$ depends on $b_n$ and $H_n$, its dimension depends only on $H_n$.
This dependence on $H_n$ shows the richness of these fractal sets, which are in the opposite direction of uniform distribution.

\begin{proof}[Proof of Lemma \ref{thm-main-ebp}]
Recall that $\tilde{f}(x)=[f(x)]$ for any scaling map $f:\Omega \mapsto \mathcal{F}$.
By $\lambda_1\geq H_0$, $f_1(\Omega) \supset \mathcal{O}$.
Thus the map $\tilde{f_1}: \Omega \rightarrow \mathcal{O}$ is a $q^{\lambda_1-H_0}$-to-1 scaling surjection.
Then the restriction map
\[\tilde{f}_{1,h}: D_{1,h} \rightarrow \mathcal{O}\]
is bijective, where $D_{1,h}$ is any disk of radius $q^{-\lambda_1}$ in $\Omega$, $h=1,\ldots,q^{\lambda_1-H_0}$.
Hence
\[D_{1,h}':= \tilde{f}_{1,h}^{-1}(b_1+\pi^{H_1}\mathcal{O})\]
is a disk of radius $q^{-\lambda_1-H_1}$ in $D_{1,h}$.
Moreover, for $H_0 \leq i<\lambda_1$, $i\in I$ and for $\lambda_1\leq i\leq H_1-1+\lambda_1$, $i\notin I$.
Denote an arbitrary $D_{1,h}'$ by $D_1$.

Inductively, for $n>1$, we have known that $D_{n-1}$ is a disk of radius $q^{-\lambda_{n-1}-H_{n-1}}$ and
\[D_{n-1} \subset \{x\in \Omega: |[f_k(x)]-b_k|_\mathfrak{p}\leq q^{-H_k},\; \forall 1\leq k\leq n-1\}.\]
Since $\lambda_n-\lambda_{n-1}\geq H_{n-1}$, the map $f_n: D_{n-1} \rightarrow \mathcal{O}$ is a $q^{\lambda_n-\lambda_{n-1}-H_{n-1}}$-to-1 scaling surjection.
Then the restriction map
\[\tilde{f}_{n,h}: D_{n,h} \rightarrow \mathcal{O}\]
is bijective, where $D_{n,h}$ is any disk of radius $q^{-\lambda_n}$ in $D_{n-1}$, $h=1,\ldots,q^{\lambda_n-\lambda_{n-1}-H_{n-1}}$.
Hence
\[D_{n,h}':= \tilde{f}_{n,h}^{-1}(b_n+\pi^{H_n}\mathcal{O})\]
is a disk of radius $q^{-\lambda_n-H_n}$ in $D_{n,h}$.
Moreover, for $H_{n-1}+\lambda_{n-1}\leq i<\lambda_n$, $i\in I$ and for $\lambda_n\leq i\leq H_n-1+\lambda_n$, $i\notin I$.
Denote an arbitrary $D_{n,h}'$ by $D_n$.

In summary, $\Gamma_{b, H}$ is $q$-homogeneous and the branched level set $I$ is exactly
\[\Z\cap\left([H_0,\lambda_1-1]\cup \bigcup_{n\geq 1}[H_n+\lambda_n,\lambda_{n+1}-1]\right).\]
As in Lemma \ref{lem-dimlub}, we have
\[m_1=q^{\lambda_1-H_0},\quad d_1=q^{-\lambda_1-H_1},\quad \delta_1=q^{1-\lambda_1},\]
\[m_n=q^{\lambda_n-\lambda_{n-1}-H_{n-1}},\quad d_n=q^{-\lambda_n-H_n},\quad \delta_n=q^{1-\lambda_n} \quad (n\geq2).\]
Note that $m_n\delta_n=qd_{n-1}$. Thus the Hausdorff dimension
\[
\dim_{\mathcal{H}}\Gamma_{b, H}= \liminf_{k\to \infty} \frac{\log (m_1\cdots m_k)}{-\log (d_k)}
= \liminf_{n\rightarrow\infty} \frac{\lambda_n-\sum_{k=1}^{n-1} H_k}{\lambda_n+H_n}.
\]
\end{proof}

Note that if the branch level set of a $q$-homogeneous set is a strictly increasing sequence $(\kappa_n)$, then the Hausdorff dimension of the $q$-homogeneous set is
\begin{equation}
\liminf_{n\rightarrow \infty} \frac{n}{\kappa_n}
\end{equation}
by Lemma \ref{lem-dimlub}. In particular, if the branch level set is of density $\rho$ in $\N$, then the Hausdorff dimension of the $p$-homogeneous set is exactly $\rho$.
(So we may not derive Corollary \ref{cor-fulldim} directly from Lemma \ref{thm-main-ebp}.)

\begin{proof}[Proof of Prop \ref{prop-fulldim2}]
Assume that $|\alpha|_\mathfrak{p}=q^{S}$.
Take an $L\in\N_+$ such that
\[L\gg \max\{H, K, |S|\}.\]
The set
\[\{x\in\mathcal{F}: |x|_\mathfrak{p}>1;\; |[\alpha x]-b_1|_\mathfrak{p}\leq q^{-H}\}\]
contains a disk
\[\Omega:=A/\alpha+b_1/\alpha+\pi^{H+S}\mathcal{O}\]
with $|A|_\mathfrak{p}=q^L$. We only need to consider the subset of $\Omega$. Let $E'=E\cap \Omega$.
By Lemma \ref{lem-scaling}, the map $x\mapsto \alpha x^n$ is a scaling map on $\Omega$ of scaling ratio $|\alpha|_\mathfrak{p}|n|_\mathfrak{p} q^{L(n-1)}$.

Take
\begin{equation}\label{eq:lem5.2(1)}
m_n=n+\lfloor \frac{n-1}{p^K-1} \rfloor
\end{equation}
and $f_n(x)=\alpha x^{m_n}$ for $n\geq 1$. Then $(m_n)$ is the sequence of all positive integers that are not divided by $p^K$.
Note that
\[E'=\{x\in \Omega: |[f_n(x)]-b_n|_\mathfrak{p}\leq q^{-H}, \forall n\geq 2\}.\]

Now the scaling ratios satisfy that
\[|\alpha|_\mathfrak{p}|m_2|_\mathfrak{p} q^{L(m_2-1)}>q^H
\quad \text{and} \quad
\frac{|\alpha|_\mathfrak{p}|m_{n+1}|_\mathfrak{p}q^{L(m_{n+1}-1)}}{|\alpha|_\mathfrak{p}|m_n|_\mathfrak{p}q^{L(m_n-1)}}>q^H.\]
By Lemma \ref{thm-main-ebp} and \eqref{eq:lem5.2(1)}, $E'$ is a $q$-homogeneous set in $D$ with
\begin{align*}
\dim_{\mathcal{H}} E'
&= \liminf_{n\rightarrow\infty} \frac{\log_q |m_n|_\mathfrak{p}+(m_n-1)L-(n-1)H}{\log_q |m_n|_\mathfrak{p}+(m_n-1)L+H}\\
&= \liminf_{n\rightarrow\infty} \left(1-\frac{nH}{m_nL}\right)\\
&= 1-(1-p^{-K})\frac{H}{L}.
\end{align*}
Since $L$ can be arbitrarily large and $E'\subset E$, $\dim_{\mathcal{H}}E=1$.
\end{proof}

There are some cases where the sequence $(\{x^n\})$ in $\R$ is not uniformly distributed. Except for some trivial cases, it seems that Pisot numbers are the only cases that can be found. Recall that an algebraic integer $\xi$ is a Pisot numbers if $\xi>1$ and the Archimedean absolute values of all its conjugate roots are smaller than 1. Still, we can talk about Pisot-Chabauty numbers, which are the $p$-adic version of Pisot numbers. This conception is from \cite{SSS2016}.

\begin{example}\label{Pisot}
\rm{ We say that $\xi \in \Q_p$ is a \textit{Pisot-Chabauty number} if
\begin{itemize}
    \item[(1)] $\xi=\xi_1$ is a root of a monic polynomial over $\Z[1/p]$ with conjugates $\xi_2,\ldots,\xi_m \in \Q_p$;
    \item[(2)] $|\xi_1|_p >1$;
    \item[(3)] $|\xi_i|_p <1$ for $i=2,\ldots,m$;
    \item[(4)] $|\xi_i|_\infty <1$ for $i=1,2,\ldots,m$.
\end{itemize}
For example, if $k>l$ are two positive integers, then $\frac{-1-\sqrt{1-4p^{k+l}}}{2p^k}$ in $\Q_p$ is a Pisot-Chabauty number.

We claim that for a Pisot-Chabauty number $\xi$, the only possible limit points of the sequence $([\xi^n])_{n\geq 1}$ are $0$ and $-1$, so that this sequence is not uniformly distributed in $\Z_p$.

Let $T_n(\xi)=\xi_1^n+\cdots+\xi_m^n$. Since $T_n(\xi)$ is a symmetric function of $\xi_1,\ldots,\xi_m$ and the coefficients of the minimal polynomial of $\xi$ are in $\Z[1/p]$, $T_n(\xi) \in \Z[1/p]$. For large $n$, $|T_n(\xi)|_\infty <1$, so
\[T_n(\xi) \in \{\pm m/p^v: v\in \N_+,m=1,\ldots,p^v-1\},\]
which means that $[T_n(\xi)]=0$ or $-1$. Now for large $n$,
\[[\xi_1^n]-[T_n(\xi)] =\xi_1^n-T_n(\xi) =\xi_2^n+\cdots+\xi_m^n,\]
and so $|[\xi_1^n]-[T_n(\xi)]|_p \rightarrow 0$.
}
\end{example}

\section{Some metrical results on the distribution}

Based on Lemma \ref{thm-main-ebp}, there are more metrical results on the distribution of the sequences $\left([\alpha x^n]\right)$ and $\left([\beta^n x]\right)$.

For the sequence of the form $\beta^n x$, the Mahler's 3/2 problem \cite{Mahler1968} is an interesting question in the case of real numbers.
The question is that whether we can find a positive number $x$ such that
\[0\leq \{x(\frac{3}{2})^n\}< 1/2\]
for all $n\geq 0$.
Similarly, can we find some $x\in \Q_2$ such that
\[[x (\frac{3}{2})^n]\in 2\Z_2\]
for all $n$?
This is much easier since $|\frac{3}{2}|_2=2$ is an integer while $|\frac{3}{2}|_\infty=1.5$ is not a rational power of any integer.
By Lemma \ref{thm-main-ebp}, for any disk $D$ of radius 1 in $\Q_2$, there is exactly one point $x\in D$ such that for all $n\geq 0$, $[x (\frac{3}{2})^n]\in 2\Z_2$.
(Note that if the branch level set is a finite set, then the $q$-homogeneous set is a finite set. In particular, if the branch level set is empty, then the $q$-homogeneous set is one point.)
\medskip

For the sequence $\left([\alpha x^n]\right)$, we have the following two propositions.
Proposition \ref{prop-other1} is a non-Archimedean version of Theorem 1.1 in \cite{Baker2015} and Proposition \ref{prop-other2} is a non-Archimedean version of Theorem 1 in \cite{BLM2019}.
We may repeat the proofs in \cite{Baker2015} and \cite{BLM2019} using non-Archimedean language, but we could simplify the proofs via Lemma \ref{thm-main-ebp}.

\begin{proposition}\label{prop-other1}
Assume ${\rm char}\mathcal{F}=0$. Fix $0\neq \alpha \in \mathcal{F}$.
Let $K\in \N_+$. Let $|z|_\mathfrak{p}>1$ and $\delta>0$. Suppose that $(r_n)$ is a strictly increasing sequence in $\N$ satisfying
\[\lim_{n\to \infty} (r_{n+1}-r_n)=\infty\]
and $p^K\nmid r_n$ for any $n$. For any sequence $(b_n)$ in $\mathcal{O}$, take
\[E(r_n,b_n):=\{x\in D(z,\delta): |x|_\mathfrak{p}>1, \lim_{n\to \infty}([\alpha x^{r_n}]-b_n)=0)\},\]
then $\dim_{\mathcal{H}}E(r_n,b_n)=1$.
\end{proposition}

\begin{proof}
Without loss of generality we may assume that $\delta$ is an integer power of $q$ and $\delta<\min\{1,|z|_\mathfrak{p}/q^{e+1}\}$.
Let $0<\epsilon<1$ be some constant. Now we take two refinements $(\tilde r_n)$, $(\tilde b_n)$ of $(r_n)$ and $(b_n)$, such that
\begin{equation}\label{eq:prop7.1(1)}
\lim_{n\to \infty} (\tilde r_{n+1}-\tilde r_n)=\infty,
\quad
\tilde r_{n+1}\leq (1+\epsilon)\tilde r_n
\quad \hbox{and} \quad
p^K\nmid \tilde r_n,\;\forall n\geq 1
\end{equation}
(we can do this as in the proof in \cite{Baker2015}). Note that $E(\tilde r_n,\tilde b_n) \subset E(r_n,b_n)$, so we only need to consider the lower bound of the Hausdorff dimension of $E(\tilde r_n,\tilde b_n)$.

Let $0<\eta<1$ be a constant. Take 
\[f_n: D(z,\delta) \rightarrow \mathcal{F} \quad f_n(x)=\alpha x^{\tilde r_{n+N-1}}\]
be a scaling map of scaling ratio
\begin{equation}\label{eq:prop7.1(2)}
q^{\lambda_n}=|\alpha|_\mathfrak{p}| \tilde r_{n+N-1}|_\mathfrak{p} |z|_\mathfrak{p}^{\tilde r_{n+N-1}-1}.
\end{equation}
Let
\begin{equation}\label{eq:prop7.1(3)}
\epsilon_n:=|z|_\mathfrak{p}^{-\lfloor (1-\eta)(\tilde r_{n+N}-\tilde r_{n+N-1}) \rfloor}.
\end{equation}
This $N$ here satisfies that
\[|\tilde r_N|_\mathfrak{p} |z|_\mathfrak{p}^{\tilde r_N-1} > \delta^{-1}|\alpha|_\mathfrak{p}^{-1}
\quad \hbox{and} \quad
\tilde r_{n+N}-\tilde r_{n+N-1} > \frac{K}{\eta \log_p |z|_\mathfrak{p}}, \ \forall n\geq N.\]
Thus $q^{\lambda_1}>\delta^{-1}$, $q^{\lambda{n+1}-\lambda_n}\geq \epsilon_n^{-1}$ and $\epsilon_n \rightarrow 0$.
By Lemma \ref{thm-main-ebp} and \eqref{eq:prop7.1(1)}-\eqref{eq:prop7.1(3)}, the set
\[E=\{x\in D(z,\delta): |[f_n(x)]-b_n|_\mathfrak{p}\leq \epsilon_n,\; \forall n>0\}\]
is a $q$-homogeneous set with
\begin{align*}
&\dim_{\mathcal{H}}E=\liminf_{n\rightarrow\infty} \frac{\lambda_n+\sum_{k=1}^{n-1} \log_q \epsilon_k}{\lambda_n-\log_q \epsilon_n} \\
&\geq \liminf_{n\rightarrow\infty} \frac{\log_q |\alpha|_\mathfrak{p}|\tilde r_{n+N-1}|_\mathfrak{p}+\bigl((\tilde r_{n+N-1}-1)-(1-\eta)\displaystyle\sum_{k=1}^{n-1} (\tilde r_{k+N}-\tilde r_{k+N-1})\bigr)\log_q |z|_\mathfrak{p}}{\log_q |\alpha|_\mathfrak{p}|\tilde r_{n+N-1}|_\mathfrak{p}+\bigl((\tilde r_{n+N-1}-1)+(1-\eta)(\tilde r_{n+N}-\tilde r_{n+N-1})\bigr)\log_q |z|_\mathfrak{p}} \\
&= \liminf_{n\rightarrow\infty} \frac{\log_q |\alpha|_\mathfrak{p}|\tilde r_{n+N-1}|_\mathfrak{p}+\bigl((\tilde r_{n+N-1}-1)-(1-\eta)(\tilde r_{n+N-1}-\tilde r_{N})\bigr)\log_q |z|_\mathfrak{p}}{\log_q |\alpha|_\mathfrak{p}|\tilde r_{n+N-1}|_\mathfrak{p}+\bigl((\tilde r_{n+N-1}-1)+(1-\eta)(\tilde r_{n+N}-\tilde r_{n+N-1})\bigr)\log_q |z|_\mathfrak{p}}\\
&= \liminf_{n\rightarrow\infty} \frac{(\tilde r_{n+N-1}-1)-(1-\eta)(\tilde r_{n+N-1}-\tilde r_{N})}{(\tilde r_{n+N-1}-1)+(1-\eta)(\tilde r_{n+N}-\tilde r_{n+N-1})}\\
&= \liminf_{n\rightarrow\infty} \frac{\eta\tilde r_{n+N-1}}{(\eta\tilde r_{n+N-1}+(1-\eta)\tilde r_{n+N})}\\
&\geq \liminf_{n\rightarrow\infty} \frac{\eta\tilde r_{n+N-1}}{(\eta+(1-\eta)(1+\epsilon))\tilde r_{n+N-1}} \\
&= \frac{\eta}{1+\epsilon-\eta\epsilon}.
\end{align*}

Let $\eta\rightarrow1$, since $E\subset E(\tilde r_n,\tilde b_n)$, $\dim_{\mathcal{H}} E(\tilde r_n,\tilde b_n)=1$. Hence $\dim_{\mathcal{H}} E(r_n,b_n) =1$.
\end{proof}

\begin{proposition}\label{prop-other2}
Assume ${\rm char}\mathcal{F}=0$. Fix $0\neq \alpha \in \mathcal{F}$.
Let $|z|_\mathfrak{p}>1$ and $0<\delta<|z|_\mathfrak{p}$. Let $\tau>1$. For any sequence $(b_n)$ in $\mathcal{O}$, take
\[E:=\{x\in D(z,\delta): |[\alpha x^n]-b_n|_\mathfrak{p}\leq \tau^{-n} \hbox{ for infinitely many }n\},\]
then
\[\dim_{\mathcal{H}}E=\frac{\log_q |z|_\mathfrak{p}}{\log_q (\tau |z|_\mathfrak{p})}.\]
\end{proposition}

\begin{proof}
Without loss of generality we may assume that $\delta$ is an integer power of $q$ and $\delta\leq |z|_\mathfrak{p}/q^{e+1}$. Note that for any positive integer $n$ and $x\in \mathcal{F}$,
\[|x|_\mathfrak{p}\leq \tau^{-n} \Leftrightarrow |x|_\mathfrak{p}\leq q^{-\lceil n\log_q \tau \rceil}.\]

Lower bound: 
Choose a subsequence $(n_k)$ of $\N$ satisfying that
\begin{equation}\label{eq:prop7.2(1)}
	p\nmid n_k,\ |z|_\mathfrak{p}^{n_1-1}\geq \delta^{-1}|\alpha|_\mathfrak{p}^{-1},\ |z|_\mathfrak{p}^{n_{k+1}-n_k}\geq \tau^{n_k} \hbox{ and } \frac{n_1+n_2+\cdots+n_k}{n_{k+1}}\rightarrow0.
\end{equation}
Let
\[E':=\{x\in D(z,\delta):|[\alpha x^{n_k}]-b_{n_k}|_\mathfrak{p}\leq q^{-\lceil n_k\log_q \tau \rceil},\; \forall k\}\subset E.\]
We only need to prove $\dim_{\mathcal{H}}E'\geq \frac{\log_q |z|_\mathfrak{p}}{\log_q (\tau |z|_\mathfrak{p})}$.
By Lemma \ref{lem-scaling} and \eqref{eq:prop7.2(1)}, the map
\[f_k:D(z,\delta)\rightarrow\mathcal{F} \quad f_k(x)=\alpha x^{n_k}\]
is a scaling map of scaling ratio $|\alpha|_\mathfrak{p}|z|_\mathfrak{p}^{n_k-1}$. By Lemma \ref{thm-main-ebp} and \eqref{eq:prop7.2(1)},
\begin{align*}
	\dim_{\mathcal{H}}E'
	&= \liminf_{k\rightarrow\infty} \frac{\log_q |\alpha|_\mathfrak{p}+(n_k-1)\log_q |z|_\mathfrak{p}-\left(\lceil n_1\log_q \tau \rceil+\cdots+\lceil n_{k-1}\log_q \tau \rceil \right)}{\log_q |\alpha|_\mathfrak{p}+(n_k-1)\log_\mathfrak{p} |z|_\mathfrak{p}+\lceil n_k\log_q \tau \rceil}\\
	&\geq \liminf_{k\rightarrow\infty} \frac{\log_q |\alpha|_\mathfrak{p}+(n_k-1)\log_q |z|_\mathfrak{p}-\left(n_1+\cdots+n_{k-1}\right)\log_q \tau-(k-1)}{\log_q |\alpha|_\mathfrak{p}+(n_k-1)\log_q |z|_\mathfrak{p}+n_k\log_q \tau+1}\\
	&= \frac{\log_q |z|_\mathfrak{p}}{\log_q (\tau|z|_\mathfrak{p})}.
\end{align*}

Upper bound:
Take $f_n(x)=\alpha x^n$, then $f_n$ is a scaling map on $D(z,\delta)$ of scaling ratio
\[q^{\lambda_n}= |\alpha|_\mathfrak{p}|n|_\mathfrak{p}|z|_\mathfrak{p}^{n-1}.\]
Choose an $n_0$ such that $\forall n>n_0$, $|n|_\mathfrak{p}|z|_\mathfrak{p}^{n-1}>\delta^{-1}|\alpha|_\mathfrak{p}^{-1}$.
For any fixed $n>n_0$, decompose the disk $D(z,\delta)$ into the union of disks $D_j$ of radius $q^{-\lambda_n}$ with $1\leq j\leq \delta q^{\lambda_n}$, so that the restriction $\tilde{f}_{n,j}:=\tilde{f}_{n}:D_j \rightarrow \mathcal{O}$ is bijective.
Then the disks
\[D_{n,j}':=\tilde{f}_{n,j}^{-1}\left(D\left(b_n,p^{-\lceil n\log_q \tau \rceil}\right)\right)\]
are of radius $q^{-\lceil n\log_q \tau \rceil-\lambda_n}$ for $j=1,\ldots,\delta q^{\lambda_n}$.
By the definition of $E$ and $D_{n,j}'$, we have
\[E\subset \bigcup_{n>n_0} \left( \bigcup_{1\leq j\leq \delta q^{\lambda_n}} D_{n,j}'\right).\]

If $\frac{\log_q |z|_\mathfrak{p}}{\log_q (\tau|z|_\mathfrak{p})}<s<1$, then
\begin{align*}
	\sum_{\substack{n>n_0 \\ 1\leq j\leq \delta q^{\lambda_n}}} ({\rm diam}D_{n,j}')^s
	&= \sum_{n>n_0} \delta q^{\lambda_n}\left(q^{-\lceil n\log_q \tau \rceil-\lambda_n}\right)^s\\
	&= \sum_{n>n_0} \delta q^{(1-s)\lambda_n}(p^{-\lceil n\log_p \tau \rceil})^s\\
	&\leq \delta |\alpha|_\mathfrak{p}^{1-s} \sum_{n>n_0} \left(|z|_\mathfrak{p}^{n-1}\right)^{1-s}(\tau^{-n})^s<\infty.
\end{align*}
Since $\sup_{n>n_0}{\rm diam}D_{n,j}'\rightarrow0$ as $n_0\rightarrow\infty$, we have $\dim_{\mathcal{H}}E\leq s$ by the definition of Hausdorff dimension. Thus $\dim_{\mathcal{H}}E\leq \frac{\log_q |z|_\mathfrak{p}}{\log_q (\tau|z|_\mathfrak{p})}$.
\end{proof}
\medskip

\bibliographystyle{siam}
\bibliography{p-adic-koksma}

\end{document}